\documentclass[11pt]{article}
\usepackage[intlimits]{amsmath}
\usepackage{amsfonts, amssymb, amsthm}
\usepackage{esint}
\usepackage{epsfig}
\usepackage{color}
\usepackage{threeparttable} 
\usepackage[toc,page]{appendix}

\usepackage{hyperref}

\usepackage{setspace}

\newtheorem{thm}{Theorem}[section]
\newtheorem{prop}[thm]{Proposition}

\newtheorem{lemma}[thm]{Lemma}
\newtheorem{defn}[thm]{Definition}

\newtheorem{preremark}[thm]{Remark}

\numberwithin{equation}{section}

\newenvironment{rcases}
  {\left.\begin{aligned}}
  {\end{aligned}\right\rbrace}

\newcommand{\R}{\mathbb R}

\newcommand{\grad} {\nabla}

\DeclareMathOperator{\supp}{supp}

\def\XXint#1#2#3{{\setbox0=\hbox{$#1{#2#3}{\int}$}
       \vcenter{\hbox{$#2#3$}}\kern-.5\wd0}}

\newcommand{\meanbar}[1]{%
\setbox0 = \hbox{$#1 \int$}
\hbox to 0pt{%
\thinspace
\hskip 0.1\wd0
\raise 0.5\ht0
\hbox{%
\lower 0.5\dp0
\hbox{\rule{0.8\wd0}{2\linethickness}}
}%
\hss
}%
}

    \newcounter{myfootertablecounter}

\begin{document}
\title{Parabolic Obstacle Problems. \\ Quasi-convexity and Regularity}

\author{Ioannis Athanasopoulos, Luis Caffarelli, Emmanouil Milakis}
\date{}
\maketitle

\begin{abstract}
In a wide class of the so called Obstacle Problems of parabolic type it is shown how to improve the optimal regularity of the solution and as a consequence how to obtain space-time regularity of the corresponding free boundary.
\\



\end{abstract}
\let\thefootnote\relax\footnotetext{Part of this work was carried out while the first and the third authors were visiting the University of Texas. They wish to thank the Department of Mathematics and the Institute for Computational Engineering and Sciences for the warm hospitality and support during several visits the last few years. I. Athanasopoulos was partially supported by ELKE. L. Caffarelli was supported by NSF grants. E. Milakis was supported by Marie Curie International Reintegration Grant No 256481 within the 7th European Community Framework Programme and NSF DMS grant 0856687.  I. Athanasopoulos wishes to thank IMPA for the invitation where the majority of the results of the present article were presented during August 2015 (available online \href{http://video.impa.br/index.php?page=analysis-partial-differential-equations}{here}).
}
\tableofcontents
\section{Introduction}\label{Intro}

\doublespacing


Obstacle problems are characterized by the fact that the solution must satisfy unilateral constraint i.e. 
must remain, on its domain of definition or part of it, above a given function the so called obstacle. 
Parabolic obstacle problems, i.e. when the involved operators are of parabolic type, can be formulated in various ways such as a system of inequalities, variational inequalities, Hamilton-Jacobi equation, etc. More precisely, as a system of inequalities, one seeks a solution $u(x,t)$ which satisfies

\begin{equation}\label{intro1}
\begin{aligned}
\begin{rcases}
& u_t+Au\geq 0, \ u  \geq \psi\ \ \  \\
& (u_t+Au)(u-\psi)  =0 \ \ \ 
\end{rcases}  & \  \ {\rm{in}} \ \  \Omega \times (0,T] \\
 u=\phi\ \ \ \ \ \ &\ \ {\rm{on}} \ \  \partial_p(\Omega \times (0,T]) 
\end{aligned}
\end{equation}
or a solution $u(x,t)$ to

\begin{equation}\label{intro2}
\begin{aligned}
u_t+Bu= 0\ \ \  & \ \ \ {\rm{in}} \ \  \Omega \times (0,T] \\
\begin{rcases}  
&u\geq \psi, \ \ \alpha u_t+u_\nu\geq 0\ \ \ \\
&(\alpha u_t+u_\nu)(u-\psi)=0\ \  \\
\end{rcases} &\ \ \ {\rm{on}} \ \  \Gamma \times (0,T] \\
u=\phi\ \ \  & \ \ \ {\rm{on}} \ \  \partial_p(\Omega \times (0,T])\setminus (\Gamma\times (0,T]) 
\end{aligned}
\end{equation}
where $A$ and $B$ are (non-negative) definite elliptic operators. Usually, (\ref{intro1}) is referred as a 
thick obstacle problem and (\ref{intro2}) with $\Gamma\subset\partial\Omega$ (when $\alpha=0$) 
as a Signorini boundary obstacle problem (or thin 
obstacle problem if one takes $\Gamma$ to be a $(n-1)-$ manifold in $\Omega$). We shall refer to (\ref{intro2}) as the dynamic thin obstacle problem if $\alpha>0$ and to nondynamic thin obstacle or Signorini problem if $\alpha=0$. Recently, there is an intense interest, 
perhaps due to the connectivity to jump or anomalous diffusion, to study (\ref{intro1}) in all of $\R^n$ when $A$ is a non-local operator and especially the 
fractional Laplacian. Observe that when $A$ is the $\frac{1}{2}-$Laplacian there is an obvious equivalence between 
(\ref{intro1}) and (\ref{intro2}) which is identified by the Neumann-Dirichlet map, provided that $B$ is minus the Laplacian, $\Gamma\subset\R^{n-1}$, 
and $\alpha=1$. This equivalence remains true for any fractional Laplacian if $B$ is replaced by an appropriate degenerate 
elliptic operator as it was introduced in \cite{CS}.

Every problem of the above mentioned ones and their obvious generalizations is actually a minimum of linear monotone operators therefore second order incremental quotients are "supersolutions" and satisfy a minimum principle. That is "for $z=(x,t)$ with $x\in\Omega$ in (1.1) or $z=(x',t)$ with $x'\in \R^{n-1}$ in (1.2)
$$u(z+w)+u(z-w)-2u(z)$$
has no interior minima". In particular, in the limit
$D_{ww} u$ cannot attain a minimum in the interior of the domain of definition and on the hyperplane in case (1.2). This means minima must occur at the initial or lateral data (minus the hyperplane in case (1.2)). Therefore for an appropriate data we have an $L^{\infty}$ bound from below. This is certainly true if the data is smooth enough or just when the data stays strictly above the obstacle (\S\ref{section3}). In fact, we believe that an appropriate barrier would give interior quasi-convexity of solutions under general data.     
 
The purpose of this work is to show that the quasi-convexity property, absent in the literature so far, has strong implications in the study of the above problems. One such implication is the improvement of the optimal time regularity i.e. we prove that the 
positive time derivative is continuous (\S\ref{section4}) for a wide class of problems. Let us mention that in the literature there are only three cases in which the time derivative is continuous and 
all three rely on the fact that the time derivative is a priori non negative. These are the one-phase Stefan problem (\cite{CFr}), the (non-dynamic) thin 
obstacle problem (\cite{A} only in $n=2$) and, very recently, the parabolic fractional obstacle problem (\cite{CFi}). 

For further implications of the quasi-convexity we concentrate on the (nondynamic) thin obstacle problem or (time dependent) Signorini problem. The other cases i.e. the dynamic parabolic obstacle problem, the nondynamic and dynamic fractional counterparts, as well as the one with parabolic nonlocal operators is a long term project and they will be treated in forthcoming papers (see \cite{ACMPart2}). Also, elsewhere we show how one can get with this approach free boundary regularity for the already known result (\cite {CPS}) of the "thick" obstacle. 
Actually, in this case, i.e. the (time dependent) Signorini problem, we prove the optimal regularity of the space derivative (\S\ref{optimal space}), as a consequence of the parabolic monotonicity formula stated in the appendix of \cite{AC}. Secondly, we prove that the regularity of the time derivative  (\S\ref{optimal time}) near free boundary points of positive parabolic density with respect to the coincidence set is as "good" as that of the space derivative; let us point out that the results in \S\ref{optimal time} are, in fact, independent of the quasi-convexity. And finally, in \S\ref{free boundary}, since \S\ref{optimal time} yields control of the speed of the free boundary, we prove (space and time) regularity of the free boundary near "non-degenerate" free boundary points.

The results of the present paper were presented by the first author in IMPA, Rio de Janeiro, August, 17- 21, 2015 during the "International Conference on Current Trends in Analysis and Partial Differential Equations". A video of the talk is available online at \href{http://video.impa.br/index.php?page=analysis-partial-differential-equations}{http://video.impa.br}.

\section{Quasi-convexity}\label{section3}

In this section we prove the quasi-convexity of the solution for a wide class of Parabolic Obstacle Problems.  In order to avoid technicalities we shall concentrate on five prototypes of this class::

\textit{\textbf{1st prototype (Thick Obstacle Problem):}} Given a bounded domain $\Omega$ in $\R^n$, a function $\psi(x,t)$ 
(the obstacle) where $\psi<0$ on $\partial \Omega\times(0,T]$, $\max \psi(x,0)>0$ and a 
function $\phi$ with $\phi=0$ on $\partial \Omega\times(0,T]$, $\phi\geq \psi$ on $\Omega\times\{0\}$, find a function $u$ such that 

\begin{equation}\label{model1}
\begin{cases}  
u_t-\Delta u\geq 0, \ \ u\geq \psi  &{\rm{in}} \ \  \Omega \times (0,T] \cr
(u_t-\Delta u)(u-\psi)=0 &{\rm{in}} \ \  \Omega \times (0,T] \cr
u=\phi &{\rm{on}} \ \  \partial_p(\Omega \times (0,T]). \cr
\end{cases}
\end{equation}

\textit{\textbf{2nd prototype (Nondynamic Thin Obstacle Problem):}} Given a bounded domain  $\Omega$ in $\R^n$ with part of its 
boundary $\Gamma\subset \partial\Omega$ that lies on $\R^{n-1}$, a function $\psi(x,t)$ (the obstacle) 
where $\psi<0$ on $(\partial \Omega\setminus\Gamma)\times(0,T]$, $\max \psi(x,0)>0$ and a 
function $\phi$ with $\phi=0$ on $(\partial \Omega\setminus\Gamma)\times(0,T]$, $\phi\geq \psi$ on $\Gamma\times\{0\}$, find a function $u$ such that 

\begin{equation}\label{model2in}
\begin{cases}  
u_t-\Delta u=0,   &{\rm{in}} \ \  \Omega \times (0,T] \cr
\partial_\nu u\geq 0, \ \ u\geq \psi &{\rm{on}} \ \  \Gamma \times (0,T] \cr
(\partial_\nu u)(u-\psi)=0  &{\rm{on}} \ \  \Gamma \times (0,T] \cr
u=\phi &{\rm{on}} \ \  \partial_p(\Omega\setminus\Gamma \times (0,T]) \cr
\end{cases}
\end{equation}
where $\nu$ is the outward normal on $\partial\Omega$.

\textit{\textbf{3nd prototype (Dynamic Thin Obstacle Problem):}} Given a bounded domain  $\Omega$ in $\R^n$ with part of its 
boundary $\Gamma\subset \partial\Omega$ that lies on $\R^{n-1}$, a function $\psi(x,t)$ (the obstacle), $\psi<0$ on $(\partial \Omega\setminus\Gamma)\times(0,T]$, $\max \psi(x,0)>0$ and a 
function $\phi$ with $\phi=0$ on $(\partial \Omega\setminus\Gamma)\times(0,T]$, $\phi\geq \psi$ on $\Gamma\times\{0\}$, find a function $u$ such that 

\begin{equation}\label{model3in}
\begin{cases}  
u_t-\Delta u=0,   &{\rm{in}} \ \  \Omega \times (0,T] \cr
\alpha\partial_tu+\partial_\nu u\geq 0, \ \ u\geq \psi &{\rm{on}} \ \  \Gamma \times (0,T] \cr
(\alpha\partial_tu+\partial_\nu u)(u-\psi)=0  &{\rm{on}} \ \  \Gamma \times (0,T] \cr
u=\phi &{\rm{on}} \ \  \partial_p(\Omega\setminus\Gamma \times (0,T]) \cr
\end{cases}
\end{equation}
where where $\alpha\in (0,1]$ and $\nu$ is the outward normal on $\partial\Omega$.

\textit{\textbf{4th prototype (Fractional Obstacle Problem):}} Given a $\psi:\R^{n-1}\times [0,\infty)\rightarrow \R$ such that $\int_{\R^{n-1}}\frac{|\psi|}{(1+|x|)^{n-1+2s}}dx'<+\infty$ 
for all $t>0$ and $\phi:\R^{n-1}\rightarrow \R$ such that $\int \frac{|\phi|}{(1+s)^{n-1+2s}}<+\infty$ for some $0<s<1$, find a function $u$ such that

\begin{equation}\label{model4}
\begin{cases}  
\partial_t u+(-\Delta)^s u\geq0, \ \ u-\psi\geq 0   &{\rm{on}} \ \  \R^{n-1} \times (0,T] \cr
(\partial_tu+(-\Delta)^s u)(u-\psi)= 0 &{\rm{on}} \ \  \R^{n-1}\times (0,T] \cr
u(x,0)=\phi(x) &{\rm{on}} \ \  \R^{n-1}. \cr
\end{cases}
\end{equation}

\textit{\textbf{5th prototype (General Nonlocal Operators):}} Assume that $\psi:\R^{n-1}\times [0,\infty)\rightarrow \R$ is given and let 
$$\mathcal{L}u:=u_t-\int_{\R^{n-1}}g'(u(y,t)-u(x,t))K(y-x)dy$$
where $g:\R\rightarrow [0,\infty)$ is a $C^2(\R)$ function such that $g(0)=0$ and $\Lambda^{-1/2}\leq g''(z)\leq \Lambda^{1/2}$, $z\in \R$ for a given constant $\Lambda>1$. The kernel $K:\R^{n-1}\setminus\{0\}\rightarrow (0,\infty)$ satisfies
\begin{equation}\label{KernelK}
\begin{cases}  
K(-x)=K(x)  &{\rm{for\ any}} \ \  x\in\R^{n-1}\setminus\{0\} \cr
\chi_{\{|x|\leq 3\}}\frac{\Lambda^{-1/2}}{|x|^{n-1+s}}\leq K(x)\leq \frac{\Lambda^{1/2}}{|x|^{n-1+s}}\ &{\rm{for\ any}} \ \  x\in\R^{n-1}\setminus\{0\}. \cr
\end{cases}
\end{equation}
Then find a function $u$ such that
\begin{equation}\label{model5}
\begin{cases}  
\mathcal{L}u\geq0, \ \ u-\psi\geq 0   &{\rm{on}} \ \  \R^{n-1} \times (0,T] \cr
(u-\psi)\mathcal{L}u= 0 &{\rm{on}} \ \  \R^{n-1}\times (0,T] \cr
u(x,0)=\phi(x) &{\rm{on}} \ \  \R^{n-1}. \cr
\end{cases}
\end{equation}
In the following theorem we prove quasi-convexity for the first, the second, the third and the fourth prototype problems. The proof for the fifth prototype problem, although similar, can be found in \cite{ACMPart2}. 
The following theorem can be stated and proved using incremental quotients as it is mentioned in the introduction, for simplicity though, we prove it for the second $t-$derivative. Notice that the corresponding space quasi-convexity is well known from the outset of the problems.
\begin{thm}\label{semiconvexity}
Suppose that in the above problems $\psi$ and $\phi$ are smooth. If $(\phi-\psi)\big|_{t=0}> 0$ then $$||(u_{tt})^-||_{\infty}\leq \max(||(\psi_{tt})||_{\infty},||\Delta^2\phi||_{\infty}).$$
If $(\phi-\psi)\big|_{t=0}\geq 0$ the same estimate holds provided that $(\partial_t\psi-(-\Delta)^s\psi)\big|_{t=0}\geq M>0$ for $s\in(0,1]$ and $M$ sufficiently large.
\end{thm}
\begin{proof}
In all four cases we use the penalization method i.e. one obtains the solution $u$ as a limit of $u^{\varepsilon}$ as $\varepsilon\rightarrow 0$, 
where $u^\varepsilon$ is a solution, in case (\ref{model1}) of the problem
\begin{equation}\label{penmodel1}
\begin{cases}  
\Delta u^\varepsilon-\partial_t u^\varepsilon =\beta_\varepsilon(u^\varepsilon-\psi^\varepsilon) &{\rm{in}} \ \  \Omega \times (0,T] \cr
u^\varepsilon=\phi+\varepsilon &{\rm{on}} \ \  \partial_p(\Omega \times (0,T]). \cr
\end{cases}
\end{equation}
or, in case (\ref{model2in}) of the problem
\begin{equation}\label{penmodel2}
\begin{cases}  
\partial_tu^\varepsilon-\Delta u^\varepsilon=0,   &{\rm{in}} \ \  \Omega \times (0,T] \cr
-\partial_\nu u^\varepsilon=\beta_\varepsilon(u^\varepsilon-\psi) &{\rm{on}} \ \  \Gamma \times (0,T] \cr
u^\varepsilon=\phi+\varepsilon &{\rm{on}} \ \  \partial_p(\Omega\setminus\Gamma \times (0,T]) \cr
\end{cases}
\end{equation}

or, in case (\ref{model3in}) of the problem
\begin{equation}\label{penmodel2}
\begin{cases}  
\partial_tu^\varepsilon-\Delta u^\varepsilon=0,   &{\rm{in}} \ \  \Omega \times (0,T] \cr
-\alpha\partial_tu^\varepsilon-\partial_\nu u^\varepsilon=\beta_\varepsilon(u^\varepsilon-\psi) &{\rm{on}} \ \  \Gamma \times (0,T] \cr
u^\varepsilon=\phi+\varepsilon &{\rm{on}} \ \  \partial_p(\Omega\setminus\Gamma \times (0,T]) \cr
\end{cases}
\end{equation}

or, in case (\ref{model4}) of the problem
\begin{equation}\label{penmodel3}
\begin{cases}  
-(-\Delta)^s u^\varepsilon-\partial_t u^\varepsilon=\beta_\varepsilon(u^\varepsilon-\psi)   &{\rm{on}} \ \  \R^{n-1} \times (0,T] \cr
u^\varepsilon(x,0)=\phi(x)+\varepsilon &{\rm{on}} \ \  \R^{n-1} \cr
\end{cases}
\end{equation}
where $\phi^\varepsilon$, $\psi^\varepsilon$ are smooth functions (with compact support in the case of the 
whole $\R^{n-1}$), $\beta_\varepsilon(s)=-e^{\frac{\varepsilon}{s-\varepsilon}}\chi_{s\leq \varepsilon}(s)$ 
with $\psi^\varepsilon\rightarrow \psi$,  $\phi^\varepsilon\rightarrow \phi$ (locally) uniformly as $\varepsilon\rightarrow 0$. 
If, now, $(\phi-\psi)\big|_{t=0}> 0$ then differentiating twice with respect to $t$ we obtain 

\begin{equation}\label{pen2model1}
\begin{cases}  
\Delta u_{tt}^\varepsilon-\partial_t u_{tt}^\varepsilon \leq \beta'_\varepsilon(u^\varepsilon-\psi^\varepsilon)(u_{tt}^\varepsilon-\psi_{tt}^\varepsilon) &{\rm{in}} \ \  \Omega \times (0,T] \cr
u_{tt}^\varepsilon=\phi_{tt} &{\rm{on}} \ \  \partial\Omega \times (0,T] \cr
u_{tt}^\varepsilon(x,0)=\Delta^2\phi(x)   &{\rm{in} } \ \ \Omega \times \{0\} \cr
\end{cases}
\end{equation}
or, 
\begin{equation}\label{penmodel2a}
\begin{cases}  
\partial_tu_{tt}^\varepsilon-\Delta u_{tt}^\varepsilon=0,   &{\rm{in}} \ \  \Omega \times (0,T] \cr
-\partial_\nu u^\varepsilon_{tt}\geq \beta'_\varepsilon(u^\varepsilon-\psi)(u_{tt}^\varepsilon-\psi_{tt}) &{\rm{on}} \ \  \Gamma \times (0,T] \cr
u_{tt}^\varepsilon(x,0)=\Delta^2\phi(x)   &{\rm{in} } \ \ \Omega \times \{0\} \cr
\end{cases}
\end{equation}
or, 
\begin{equation}\label{penmodel3a}
\begin{cases}  
\partial_tu_{tt}^\varepsilon-\Delta u_{tt}^\varepsilon=0,   &{\rm{in}} \ \  \Omega \times (0,T] \cr
-\alpha\partial_tu_{tt}^\varepsilon-\partial_\nu u^\varepsilon_{tt}\geq \beta'_\varepsilon(u^\varepsilon-\psi)(u_{tt}^\varepsilon-\psi_{tt}) &{\rm{on}} \ \  \Gamma \times (0,T] \cr
u_{tt}^\varepsilon(x,0)=\Delta^2\phi(x)   &{\rm{in} } \ \ \Omega \times \{0\} \cr
\end{cases}
\end{equation}
or
\begin{equation}\label{penmodel4a}
\begin{cases}  
-(-\Delta)^s u_{tt}^\varepsilon-\partial_t u_{tt}^\varepsilon=\beta'_\varepsilon(u^\varepsilon-\psi)(u_{tt}^\varepsilon-\psi_{tt})   &{\rm{on}} \ \  \R^{n-1} \times (0,T] \cr
u_{tt}^\varepsilon(x,0)=\Delta^2\phi &{\rm{on}} \ \  \R^{n-1}. \cr
\end{cases}
\end{equation}
To finish the proof, apply the minimum principle to $u_{tt}^\varepsilon$.

If, on the other hand, $(\phi-\psi)\big|_{t=0}\geq 0$, following the steps above, we notice that since $||\beta||_{\infty}<+\infty$ and $\beta'\geq 0$ it is enough to have $(\partial_t\psi-(-\Delta)^s\psi)\big|_{t=0}\geq M>0$ for $s\in(0,1]$ and $M$ sufficiently large.
\end{proof}




\section{A general implication on the optimality of the time derivative}\label{section4}
In this section we show that the quasi-convexity property obtained in the last section improves the time regularity for a wide class of problems. More precisely, we prove that the positive time derivative of the solution is always continuous for this class. Our approach will be as follows: we penalize the problems, we subtract the obstacle from the solution, then we differentiate with respect to time and we work with the derived equations. We shall obtain then a global uniform modulus of continuity independent of $\varepsilon$, which will yield in the limit the desired result.

In order to avoid having a lengthy paper, in the present section we concentrate on the first three prototype problems stated in (\S\ref{section3}). The fourth and the fifth prototype problems are treated in \cite{ACMPart2}.

\subsection{The "thick" obstacle problem}\label{subsection4.1}

In this situation the derived problem takes the form:
\begin{equation}\label{4.1.2} 
\begin{cases}
\Delta v^\varepsilon -\partial_tv^\varepsilon
=\beta'_\varepsilon(u^\varepsilon
-\psi^\varepsilon)v^\varepsilon+f_t
& {\rm{in}} \ Q:=\Omega\times (0,T]\\
v^\varepsilon=(\phi^\varepsilon-\psi^\varepsilon)_t & {\rm{on}} \ \partial_p(\Omega\times (0,T])\\
v^\varepsilon=\Delta(\phi^\varepsilon-\psi^\varepsilon) & {\rm{on}} \ \Omega\times \{0\}.
\end{cases}
\end{equation}
where $v^\varepsilon=(u^\varepsilon-\psi^\varepsilon)_t$ and  $f=-(\Delta \psi^\varepsilon-\partial_t\psi^\varepsilon)$.

Our method, which uses the approach of \cite{CaffEvans}, is essentially that of DeGiorgi's, first appeared in his 
celebrated work \cite{DeG}. To simplify matters we start with a normalized situation i.e. we assume that our solution 
is between zero and one in the unit parabolic cylinder. We will prove (Proposition \ref{Prop4.1.5}) that if at the top 
center $v^\varepsilon$ is zero then in a concentric subcylinder into the future $v^\varepsilon$ decreases. 
Then we rescale and repeat. But before that we need several lemmata. Our first lemma asserts that if $v^\varepsilon$ 
is "most of the time" very near to its positive maximum in some cylinder, then in a smaller cylinder into the future $v^\varepsilon$ is strictly positive.   

\begin{lemma}\label{lemma4.1.1}
Let $Q_1(x_0,t_0)\subset Q$ where $Q_1(x_0,t_0):=B_1(x_0,0)\times (t_0-1,t_0]$ with $B_1:=\{x\in \R^n:|x-x_0|\leq 1\}$. 
Suppose that $0<v^\varepsilon<1$ in $Q_1(x_0,t_0)$ where $v^\varepsilon$ is a solution to (\ref{4.1.2}), then there 
exists a constant $\sigma>0$, independent of $\varepsilon$, such that 
\begin{equation}\label{4.1.3}
\int_{Q_1(x_0,t_0)}(1-v^\varepsilon)^2dx<\sigma
\end{equation}
implies that $v^\varepsilon\geq 1/2$ in $Q_{1/2}(x_0,t_0)$.
\end{lemma}
\begin{proof} 
For simplicity we drop the $\varepsilon$, we shift $(x_0,t_0)$ to $(0,0)$, and write $Q_1$ for $Q_1(0,0)$.
First, we derive an 
energy inequality suited to our needs. Therefore we set $w=1-v$ and the equation becomes 
$$\Delta w-\partial_tw=\beta'(u-\psi)(w-1)-f_t.$$
Choose a smooth cutoff function $\zeta$ vanishing near the parabolic boundary of $Q_1$ and $k\geq 0$. Multiply 
the above equation by $\zeta^2(w-k)^+$ and integrate by parts to obtain
$$\frac{1}{2}\int_{Q_1}\partial_t[(\zeta(w-k)^+)^2]dxdt+\int_{Q_1}|\grad(\zeta(w-k)^+)|^2dxdt =\int_{Q_1}\beta_t(u-\psi)\zeta^2(w-k)^+dxdt$$
\begin{equation}\label{sec4-6}
+\int_{Q_1}[(w-k)^+]^2(|\grad \zeta|^2+\zeta\partial_t\zeta)dxdt+\int_{Q_1}f_t\zeta^2(w-k)^+dxdt.
\end{equation}
Integrating by parts in $t$ the last term on the right we obtain
$$\frac{1}{2}\int_{B_1}(\zeta(w-k)^+)^2(x,0)dx +\int_{Q_1}|\grad(\zeta(w-k)^+)|^2dxdt =-\int_{Q_1}\beta(u-\psi)\partial_t(\zeta^2(w-k)^+)dxdt$$
\begin{equation}\label{sec4-7}
+\int_{B_1}\beta(u-\psi)\zeta^2(w-k)^+(x,0)dx+\int_{Q_1}[(w-k)^+]^2(|\grad \zeta|^2+\zeta\partial_t\zeta)dxdt+\int_{Q_1}f_t\zeta^2(w-k)^+dxdt.
\end{equation}

Note, since $\beta$ is nonpositive and that the upper limit of $t-$integration, $t=0$, could have been replaced by any $-1\leq t\leq 0$, 
our energy inequality takes the form
$$\max_{-1\leq t\leq 0}\int_{B_1}(\zeta(w-k)^+)^2dx+\int_{Q_1}|\grad(\zeta(w-k)^+)|^2dxdt \leq $$
$$C\int_{Q_1}\bigg([(w-k)^+]^2(|\grad \zeta|^2+|\partial_t\zeta|)+(w-k)^+(|\partial_t\zeta|+1)+\chi_{\{w>k\}}\bigg)dxdt$$
with $C=2\max\{1,||\beta||_{\infty}(2+||(u_{tt})^-||_{\infty}),||f_t||_{\infty}\}$, where we have used the time quasiconvexity of the solution $u$. 

Now, we want to obtain an iterative sequence of inequalities; thus we define for $m=0,1,2,...$ 
$$k_m:=\frac{1}{2}\bigg(1-\frac{1}{2^m}\bigg), \ \ \ R_m:=\frac{1}{2}\bigg(1+\frac{1}{2^m}\bigg)$$
$$Q_m:=\bigg\{(x,t):|x|\leq R_m,\ -R_m^2\leq t\leq 0\bigg\}$$
and the smooth cutoff functions
$$\chi_{Q_{m+1}}\leq \zeta_m\leq \chi_{Q_m}$$
with $$|\grad \zeta_m|\leq C2^m,\ \ |\partial_t\zeta_m|\leq C4^m.$$
Substituting $\zeta=\zeta_m$ and setting $w_m=(w-k_m)^+$ we obtain, by the Sobolev inequality, that
$$\bigg(\int_{Q_m}(\zeta_mw_m)^{2\frac{n+2}{n}}dxdt\bigg)^{\frac{n}{n+2}}\leq C\bigg( 4^mC\int_{Q_m}w_m^2dxdt+C4^m\int_{Q_m}w_mdxdt+|Q_m\cap\{w_m\neq 0\}|\bigg)$$
$$\leq C\bigg( 4^mC\int_{Q_m}w_m^2dxdt+\bigg(\frac{4^m}{2}+1\bigg)\bigg|Q_m\cap\{w_m\neq 0\}\bigg|\bigg).$$
Since $$(k_m-k_{m-1})^2|Q_m\cap\{w_m\neq 0\}|\leq \int_{Q_m}w^2_{m-1}dxdt$$
we obtain
\begin{eqnarray}\label{sec4-9}
\int(\zeta_mw_m)^2dxdt&\leq& \bigg(\int(\zeta_mw_m)^{2\frac{n+2}{n}}dxdt\bigg)^{\frac{n}{n+2}}\bigg|Q_m\cap\{w_m\neq 0\}\bigg|^{\frac{2}{n+2}}\nonumber\\
&\leq& C16^m\bigg(\int(\zeta_{m-1}w_{m-1})^2dxdt\bigg)^{\frac{n+4}{n+2}}.
\end{eqnarray}
Setting $$I_m:=\int(\zeta_mw_m)^2dxdt$$
then they satisfy the recursive inequality
$$I_m\leq C16^mI_{m-1}^{1+\frac{2}{n+2}}.$$
It is well known from DeGiorgi's work (see for instance Lemma II.5.6, page 95 of \cite{LSU}) that $I_m\rightarrow 0$ as $m\rightarrow 0$ provided that 
$$I_0\leq \frac{1}{2^{(n+2)^2}C^{\frac{n+2}{n}}}
=:\sigma.$$ 
\end{proof}

Our second Lemma asserts that if $v^\varepsilon$ is very tiny "most of the time" in some cylinder, then in a smaller concentric cylinder, $v^\varepsilon$ goes down to $1/2$. The fact that $\beta'>0$ renders $v^\varepsilon$, more so any nonnegative solution to (\ref{model1}), a subsolution (subcaloric).

\begin{lemma}\label{lemma4.1.2}
Let $Q_1$ be as in Lemma \ref{lemma4.1.1}. Suppose that $v^\varepsilon$ is a subsolution to (\ref{4.1.2}) and that $0<v^\varepsilon<1$ in $Q_1$. Then there exists a constant $\bar{\sigma}>0$, independent of $\varepsilon$, such that 
$$\int_{Q_1}(v^\varepsilon)^2dxdt<\bar{\sigma}$$ 
implies that $v^\varepsilon\leq 1/2$ in $Q_{1/2}$.
\end{lemma}
\begin{proof}
It is identical to the proof of Lemma \ref{lemma4.1.1} except for the energy inequality which is in fact much simpler. As before we drop the $\varepsilon$. 
We see that
\begin{equation}\label{4.1.5}
\Delta v-\partial_t v\geq f_t \ \ \ {\rm{in}}\ \ \ Q_1.
\end{equation}
Therefore we multiply the equation by $\zeta^2(v-k)^+$ where $\zeta$ and $k$ are as in the proof of Lemma \ref{lemma4.1.1} and integrate by parts to obtain the energy inequality
$$\max_{-1\leq t\leq 0}\int_{B_1}(\zeta(v-k)^+)^2dx+\int_{Q_1}|\grad(\zeta(v-k)^+)|^2dxdt \leq 2\int_{Q_1}[(v-k)^+]^2(|\grad \zeta|^2+|\partial_t\zeta|)dxdt.$$
Again, we substitute $\zeta=\zeta_m$ and we set $v_m=(v-k_m)^+$ where $\zeta_m$ and $k_m$ are as in Lemma \ref{lemma4.1.1}. By Sobolev inequality

$$\int(\zeta_mv_m)^{2\frac{n+2}{n}}dxdt\leq C4^m\int v_m^2dxdt$$
and since $$(k_m-k_{m-1})^2|Q_m\cap\{v_m\neq 0\}|\leq \int v^2_{m-1}dxdt$$
we obtain
\begin{eqnarray}\label{sec4-10}
\int(\zeta_mv_m)^2dxdt&\leq& \bigg(\int(\zeta_mv_m)^{2\frac{n+2}{n}}dxdt\bigg)^{\frac{n}{n+2}}\bigg|Q_m\cap\{v_m\neq 0\}\bigg|^{\frac{2}{n+2}}\nonumber\\
&\leq& C16^m\bigg(\int(\zeta_{m-1}v_{m-1})^2dxdt\bigg)^{\frac{2}{n+2}}.
\end{eqnarray}
Hence, if $$I_m:=\int(\zeta_mw_m)^2dxdt$$
we have
$$I_m\leq C16^mI_{m-1}^{1+\frac{2}{n+2}}$$
i.e. $I_m\rightarrow 0$ as $m\rightarrow 0$ provided that 
$$I_0\leq \frac{1}{2^{(n+1)^2}C^{\frac{n+2}{2}}}=:\bar{\sigma}.$$ 
\end{proof}

The next lemma is the parabolic version of DeGiorgi's isoperimetric lemma. One version of this lemma is proved in \cite{CaffVass} and with proper adjustments applies to our situation. We state it as our next lemma.

\begin{lemma}\label{lemma4.1.3}
Given $\epsilon_1>0$, there exists a $\delta_1>0$ such that for every subsolution $v^\varepsilon$ to (\ref{4.1.2}) satisfying $0<v^\varepsilon<1$ in $Q_1$, $$|\{(x,t)\in Q_1:v^\varepsilon=0\}|\geq \sigma_0|Q_1|$$ if $$|\{(x,t)\in Q_1:0<v^\varepsilon<1/2\}<\delta_1|Q_1|$$ 
then 
$$\int_{Q_{R^{'}}}[(v^\varepsilon-\frac{1}{2})^+]^2dxdt\leq C\epsilon_1.$$
where $R^{'}=c\sigma_0$ for $\sigma_0>0$ and some $0<c<1$.  
\end{lemma}

In order to achieve our decay estimate to zero we shall take a point
$v^\varepsilon(0,0)=0$ 
at the top center of $Q_1$ and show that in $Q_{R^{'}}(0,0)$, for some $R^{'}<1$,  $v^\varepsilon$ is pointwise strictly less than one . This is the content of our next lemma.

\begin{lemma}\label{lemma4.1.4}
Let $Q_1$ and $\sigma$ be as in Lemma \ref{lemma4.1.1}. Suppose that $v^\varepsilon$ is a solution to (\ref{4.1.2}) such that $v^\varepsilon(0,0)=0$ and $0 \leq v^\varepsilon\leq 1$ in $Q_{1}$, then $v^\varepsilon\leq 1-C\sigma$ in $Q_{R^{'}}(0,0)$ where $C$ is independent of $\varepsilon$ and $R^{'}=\frac{\sigma}{8}$.
\end{lemma}
\begin{proof}
Again, we drop $\varepsilon$. 
Since $v(0,0)=0$, by Lemma \ref{lemma4.1.1}
$$\int_{Q_1}(1-v)^2dxdt\geq \sigma.$$
It follows then that 
$$|\{v<1-\frac{\sigma}{4}\}\cap Q_1|\geq \frac{1}{4}\sigma|Q_1|.$$
Therefore, we set $$w:=\frac{4}{\sigma}\bigg[v-(1-\frac{\sigma}{4})\bigg]^+$$
and we see that $w$ is a subsolution to (\ref{4.1.2}). Following DeGiorgi's method we will consider a dyadic sequence of normalized truncations $$w_k:=2^k\bigg[w-(1-2^{-k})\bigg]^+$$
still subsolutions to (\ref{4.1.2}). We will show that, in finite number of steps $k_0=k_0(\delta_1)$ (where $\delta_1$ is defined in Lemma \ref{lemma4.1.3} and $C\epsilon_1\leq \bar{\sigma}$, $\bar{\sigma}$ that of Lemma \ref{lemma4.1.2}), 
$$|\{w_{k_0}>0\}|=0.$$
Note that for every $k$, $0\leq w_k\leq 1$ and $|\{w_k=0\}\cap Q_1|\geq \frac{\sigma}{4}|Q_1|$. Assume, now, that for every $k$ $|\{0<w_k<1/2\}\cap Q_1|\geq \delta_1|Q_1|$. Then for every $k$
$$|\{w_k=0\}|=|\{w_{k-1}=0\}|+|\{0<w_k<1/2\}|\geq |\{w_{k-1}=0\}|+\delta_1|Q_1|.$$
Therefore after a finite number of steps say $k_0>1/\delta_1$ we get $|\{w_{k_0}=0\}|\geq |Q_1|$. Thus $w_{k_0}<0$ i.e. $2^{k_0}[w-(1-2^{-k_0})]^+=0$ or $w<1-2^{-k_0}$. Suppose, now, that there exists $k'$, $0\leq k'\leq k_0$ such that $$|\{0<w_{k'}<\frac{1}{2}\}|<\delta_1|Q_1|.$$
By Lemma \ref{lemma4.1.3} applied to $w_{k'}$ with $\sigma_0=\frac{\sigma}{4}$ and consequently by Lemma \ref{lemma4.1.2} applied to $w_{k'+1}$ we obtain $w_{k'+1}\leq 1/2$ in $Q_{R^{'}}$, where $R^{'}=\frac{\sigma}{8}$, i.e. $w<1-2^{-(k'+2)}$. Hence in both cases $w<1-2^{-(k_0+2)}$ in $Q_{R^{'}}$ or $v<1-2^{-(k_0+4)}\sigma$.
\end{proof}

The estimates obtained above are all independent of $\varepsilon$. We would like to iterate the lemmata above to force the maximum of $v^\varepsilon$ to decrease to zero in a dyadic sequence of decreasing parabolic cylinders in order to obtain the continuity of $v^\varepsilon$. 

\begin{prop}\label{Prop4.1.5}
Let $v^\varepsilon$ be a solution to (\ref{4.1.2}) in $Q$ then
$$|(v^\varepsilon)^+(x,t)-(v^\varepsilon)^+(x_0,t_0)|\leq C\omega(|x-x_0|^2+|t-t_0|)$$
for any $(x,t)$ and $(x_0,t_0)$ in $Q$, where $C$ is independent of $\varepsilon$ and $\omega$ denotes the modulus of continuity. 
\end{prop}
\begin{proof}
It is enough to consider only the case when $(v^\varepsilon)^+(x_0,t_0)=0$, since, otherwise, $v^\varepsilon$ satisfies a nice equation with smooth data and with regular boundary. Therefore, for simplicity, we take $(x_0,t_0)=(0,0)$ and $Q_1$ as before. Again, we drop the $\varepsilon$ and we set 
$$Q_k:=Q_{R_k}, \ \ \ M_k:=\sup_{Q_k}v$$
where $R_k
:=\frac{\sigma}{8}M_k$ and $$\bar{v}:=\frac{v_k}{M_k}$$
where $v_k(x,t):=v(R_kx,(R_k)^2t)$. Then $\bar{v}$ satisfies $$\Delta \bar{v}-\partial_t \bar{v}\geq \bar{f}_t\ \ \ {\rm{in}}\ \ \ Q_1.$$
Therefore by Lemma \ref{lemma4.1.4}, $$\sup_{Q_{R'}}\bar{v}\leq 1-C\sigma$$
or in our original setting
$$\sup_{Q_{k+1}}v\leq \mu_k\sup_{Q_k}v$$
where $\mu_k=1-C(\sup_{Q_k}v^+)^{1+\frac{n}{2}}$. So, even, if $\mu_k\rightarrow 1$ as $k\rightarrow \infty$, $M_k\rightarrow 0$. 

To finish the proof, we use a standard barrier argument to get the continuity from the future.
\end{proof}

\begin{thm}\label{theorem4.1.6}
Let $u$ be a solution to (\ref{model1}) then $(u-\psi)_t^+$ is continuous. 
\end{thm}
\begin{proof}
It is well known that a subsequence of $v^\varepsilon$ will converge uniformly to the unique solution of (\ref{model1}).
\end{proof}

\subsection{The (nondynamic) thin obstacle problem or Signorini Problem}\label{subsection4.2}

%
%
%

Let us extend $\psi^\varepsilon$ to all $\Omega$ i.e. we take any function $\tilde{\psi}^\varepsilon(x',x_n,t)$ such that $\tilde{\psi}^\varepsilon(x',0,t)=\psi^\varepsilon(x',t)$, $\tilde{\psi}^\varepsilon(x',x_n,t)<\phi$ on $\partial_p((\Omega\setminus\Gamma)\times(0,T])$ and $\frac{\partial\tilde{\psi}^\varepsilon}{\partial\nu}(x',0,t)=0.$ 
Then our problem takes the form
\begin{equation}\label{4.2.2}
\begin{cases}  
\Delta v^\varepsilon-\partial_tv^\varepsilon=f_t &{\rm{in}} \ \  \Omega \times (0,T] \cr
-\partial_\nu v^\varepsilon=\beta_\varepsilon'(u^\varepsilon
-\tilde{\psi}^\varepsilon)v^\varepsilon &{\rm{on}} \ \  \Gamma \times (0,T] \cr
v^\varepsilon=(\phi^\varepsilon-\tilde{\psi}^\varepsilon)_t &{\rm{on}} \ \  \partial_p(\overline{\Omega}\setminus\Gamma \times (0,T]) \cr
v^\varepsilon=\Delta (\phi^\varepsilon-\tilde{\psi}^\varepsilon) &{\rm{on}} \ \  \Omega \times \{0\}. \cr
\end{cases}
\end{equation}
where $v^\varepsilon=(u^\varepsilon-\tilde{\psi}^\varepsilon)_t$ and $f=-(\Delta\tilde{\psi}^\varepsilon-
\partial_t\tilde{\psi}^\varepsilon)$.


We shall repeat the approach of Section \ref{subsection4.1} but, 
instead of parabolic cylinders, we take parabolic rectangular cylinders with one of its sides lying on $\Gamma$. 
We normalize again i.e. our solution is between zero and one and we prove that, if $v^\varepsilon$ 
is zero on the top center and on $\Gamma$ in such a cylinder, then in a concentric subcylinder into the future $v^\varepsilon$ 
is smaller than one. Then we rescale and repeat. 


Our first lemma asserts that if $v^\varepsilon$ is "most of the time" very near to its positive maximum in some cylinder sitting in $(\R^{n+1})$ against the hyperplane $x_n=0$ and going backwards in time then in a smaller cylinder into the future, $v^\varepsilon$ is strictly positive.

\begin{lemma}\label{lemma4.2.1}
Let $Q_1(x_0',0,t_0)\subset \Omega\times(0,T]$ 
where $Q_1(x_0',0,t_0)=B_1(x_0',0)\times(t_0-1,t_0]$, $B_1(x_0',0)=B_1'(x'_0)\times(0,1)$, $B'_1(x_0)=\{x':|x'-x_{0}'|<1\}$ and $Q_1'(x_0',t_0)=B_1'(x_0')\times(t_0-1,t_0]$. Suppose that $0<v^\varepsilon<1$ in $Q_1(x_0,t_0)$ 
where $v^\varepsilon$ is a solution to (\ref{4.2.2}). Then there exists a constant $\sigma>0$,
 independent of $\varepsilon$, such that $$\fint_{Q'_1(x'_0,t_0)}\chi_{\{1-v^\varepsilon>0\}}dx'dt+\fint_{Q_1(x_0,t_0)}(1-v^\varepsilon)^2dxdt<\sigma$$
implies that $$v^\varepsilon\geq \frac{1}{2}$$ in $Q_{1/2}(x_0,t_0)$.
\end{lemma}
\begin{proof}
For simplicity we drop the superscript $\varepsilon$, shift $(x_0,0,t_0)$ to $(0,0,0)$ and write $Q_1$ for $Q_1(0,0,0)$. We first derive an energy inequality associated to our problem. Set $w=1-v$ then the problem becomes
\begin{equation}\label{w4.2.2}
\begin{cases}  
\Delta w-\partial_tw=-f_t &{\rm{in}} \ \  \Omega \times (-T,T] \cr
\partial_\nu w=\beta_t(u-\tilde{\psi}) &{\rm{on}} \ \  \Gamma \times (-T,T] \cr
w=1-(\phi-\tilde{\psi})_t &{\rm{on}} \ \  \partial_p(\overline{\Omega}\setminus\Gamma \times (-T,T]) \cr
w=1-\Delta (\phi-\tilde{\psi}) &{\rm{on}} \ \  \Omega \times \{-T\}. \cr
\end{cases}
\end{equation}
Choose a smooth cutoff function $\zeta$ vanishing near the parabolic boundary of $Q_1$ except on $Q_1'$ and $k\geq 0$. Multiply the above by $\zeta^2(w-k)^+$ and integrate by parts to obtain
$$\int_{Q_1}[\grad(\zeta^2(w-k)^+)\grad w+\zeta^2(w-k)^+\partial_tw]dxdt=\int_{Q_1'}\zeta^2(w-k)^+\partial_\nu wdx'dt+\int_{Q_1}f_t\zeta^2(w-k)^+dxdt$$
and
$$\int_{Q_1}\bigg[\frac{1}{2}\partial_t[(\zeta(w-k)^+)^2]+|\grad(\zeta(w-k)^+)|^2\bigg]dxdt=\int_{Q_1'}\partial_t[\zeta^2(w-k)^+\beta(u-\psi)]dx'dt$$
$$
-\int_{Q_1'}\zeta^2\partial_t(w-k)^+\beta(u-\psi)dx'dt+2\int_{Q_1'}\zeta\partial_t\zeta(w-k)^+\beta(u-\psi)dx'dt
$$
\begin{equation}\label{p5-1}
+\int_{Q_1}(|\grad \zeta|^2+\zeta\partial_t\zeta)[(w-k)^+]^2dxdt+\int_{Q_1}f_t\zeta^2(w-k)^+dxdt.
\end{equation}
Now, using the fact that $\beta$ is bounded and negative, $(u-\psi)_{tt}$ is bounded below and since the upper limit of the $t-$integration $t=0$ can be replaced by any $-1^2\leq t\leq 0$, we obtain
$$\frac{1}{2}\max_{-1\leq t\leq 0}\int_{B_1}[(w-k)^+\zeta]^2dx+\int_{Q_1}|\grad ((w-k)^+\zeta)|^2dxdt $$
$$\leq ||\beta||_{\infty}||(u-\psi)_{tt}^-||_{\infty}\int_{Q_1\cap\{w>k\}}\zeta^2dx'dt+||\beta||_{\infty}\int_{Q_1'}(w-k)^+\partial_t\zeta dx'dt$$
$$+\int_{Q_1}[(w-k)^+]^2(|\grad\zeta|^2+\partial_t\zeta)dxdt+\int_{Q_1}f_t\zeta^2(w-k)^+dxdt$$
and, a fortiori, we have the "energy inequality" 
$$\max_{-1\leq t\leq 0}\int_{B_1}[(w-k)^+\zeta]^2dx+\int_{Q_1}|\grad ((w-k)^+\zeta)|^2dxdt$$
\begin{equation}\label{eq_4.21}
\leq C\bigg(\int_{Q_1'}(\partial_t\zeta^2(w-k)^++\chi_{\{w>k\}})dx'dt+\int_{Q_1}[(w-k)^+]^2(|\grad\zeta|^2+\partial_t\zeta)dxdt+\int_{Q_1}\zeta^2(w-k)^+dxdt\bigg)
\end{equation}
where $C=\overline{C}\max\{1,||\beta||_{\infty}(2+||(u_{tt})^-||_{\infty}),||f_t||_{\infty}\}$.


Now that we have our energy inequality we shall obtain an iterative sequence of inequalities. We, therefore, define 
$$k_m=\frac{1}{2}\bigg(1-2^{-m}\bigg)\ \ \ R_m=\frac{1}{4}\bigg(1+\frac{1}{2^m}\bigg)$$
$$Q_m:=B_m\times (R^2_m,0], \ \ B_m:=B'_m\times [0,R_m], \ \ B'_m:=B'_{R_m}=\{|x'|<R_m\}, $$
$$Q_m':=\{(x_1,...,x_n,t):-R_m\leq x_i\leq R_m, -R_m\leq t\leq 0\}$$
and we choose smooth cutoff functions $\zeta_m$ such that $\chi_{Q_{m+1}}\leq \zeta_m\leq \chi_{Q_{m}}$, $|\grad\zeta_m|\leq C2^m$ and $0\leq \partial_t \zeta_m\leq C4^m$. We set $w_m=(w-k_m)^+$ and we denote by
$$I_m:=\max_{-R^2\leq t\leq 0}\int(\zeta_mw_m)^2dx+\int|\grad(\zeta_mw_m)|^2dxdt.$$
We want to prove that for every $m\geq 0$, $I_m\leq \alpha_0M^{-m}$ with $\alpha_0>0$ and $M>1$ to be chosen.  The proof is by induction, for $1\leq m\leq 2$ we choose $\sigma$ such that $4C\sigma<M^{-2}$ and for $m\geq 3$ we have
\begin{eqnarray}
I_m &\leq & C16^m\bigg(\int(w_{m-1}\zeta_{m-1})^2dxdt
+\int(w_{m-1}\zeta_{m-1})^2dx'dt\bigg)\nonumber\\
&=&  C16^m\bigg(\int(w_{m-1}\zeta_{m-1})^2dxdt
-2\int(w_{m-1}\zeta_{m-1})(w_{m-1}\zeta_{m-1})_{x_n}dxdt\bigg)\nonumber\\
&\leq & C16^m\bigg[\int(w_{m-1}\zeta_{m-1})^2dxdt
+\bigg(\int(w_{m-1}\zeta_{m-1})^2dxdt\bigg)^{1/2}\bigg(\int|\grad(\zeta_{m-1}w_{m-1})|^2dxdt\bigg)^{1/2}\bigg]\nonumber\\
\end{eqnarray}
where we used the divergence theorem and H\"{o}lder's inequality. Now, by Sobolev's inequality, we obtain
$$\int_{Q_m}(w_{m-1}\zeta_{m-1})^2dxdt\leq \bigg(\int_{Q_{m-1}}(w_{m-1}\zeta_{m-1})^{2\frac{n+2}{n}}dxdt\bigg)^{\frac{n}{n+2}}\bigg(\int_{Q_{m-1}}\chi_{\{w_{m-1}\neq 0\}}dxdt\bigg)^{\frac{2}{n+2}}$$
$$\leq 2^{4m}I_{m-2}^{1+\frac{2}{n}}.$$
Therefore, by substituting in the above we obtain
$$I_m\leq C2^{8m}(I_{m-2}^{1+\frac{2}{n}}+I_{m-1}^{1/2}I_{m-2}^{\frac{1}{2}+\frac{1}{n}}).$$
Hence, if we choose $M=2^{8n}$ and $\alpha_0=C^{-\frac{n}{2}}2^{-8n(n+2)}$, the claim is proved.

\end{proof}

From this point on we observe that, since $\beta'>0$, the boundary integral is nonnegative and can be omitted; then by reflecting about the hyperplane we are in the same situation as that of Section \ref{subsection4.1} with square cylinders. Therefore we have arrived at out main result:

\begin{prop}\label{Prop4.8}
Let $v^\varepsilon$ be a solution to (\ref{4.2.2}) in $Q$ then
$$|(v^\varepsilon)^+(x,t)-(v^\varepsilon)^+(x_0,t_0)|\leq C(|x-x_0|^2+|t-t_0|)^\alpha$$
for any $(x,t)$ and $(x_0,t_0)$ in $Q$ and some $0<\alpha<1$, where $C$ and $\alpha$ are independent of $\varepsilon$. 
\end{prop}
\begin{proof}
It is enough to consider only the case when $(v^\varepsilon)^+(x_0,t_0)=0$. For simplicity, we take $(x_0,t_0)=(0,0)$ and $Q_1$ as before. Again, we drop the $\varepsilon$ and we set 
$$Q_k:=Q_{R_k}, \ \ \ M_k:=\sup_{Q_k}v$$
where $R_k
:=\frac{\sigma}{8}M_k$ and $$\bar{v}:=\frac{v_k}{M_k}$$
where $v_k(x,t):=v(R_kx,(R_k)^2t)$. Then $\bar{v}$ satisfies $$\Delta \bar{v}-\partial_t \bar{v}\geq \bar{f}_t\ \ \ {\rm{in}}\ \ \ Q_1$$
and $$\sup_{Q_{R'}}\bar{v}\leq 1-C\sigma$$
or in our original setting
$$M_{k+1}\leq \mu_kM_k$$
where $\mu_k=1-C(\frac{M_k}{R_k})^{1+\frac{n}{2}}$. So, even, if $\mu_k\rightarrow 1$ as $k\rightarrow \infty$, $M_k\rightarrow 0$. As a matter of fact $M_k\sim 2^{-k}$ and $R_k\sim 2^{-k}$.

To finish the proof, we use a standard barrier argument to get the H\"{o}lder continuity.
\end{proof}

\begin{thm}\label{theorem4.9}
Let $u$ be a solution to (\ref{model2in}), then $(u-\psi)_t^+$ is H\"{older} continuous. 
\end{thm}
\begin{proof}
It is well known that a subsequence of $v^\varepsilon$ will converge uniformly to the unique solution of (\ref{model2in}).
\end{proof}

\subsection{The dynamic thin obstacle problem}\label{subsection4.2b}

Given a bounded domain $\Omega$ in $\R^n$ with part of its boundary $\Gamma\subset \partial\Omega$ 
lying on $\R^{n-1}$, a function $\psi$ (the obstacle) defined on $\Gamma$ such that $\max_{\Gamma\times \{0\}}\psi(x',0)>0$, $\psi(x',t)<0$ 
for every $x'\in \partial \Gamma\times(0,T]$ and a function $\phi$, 
with $\phi=0$ on $(\partial\Omega\setminus \Gamma)\times(0,T]$, $\phi\geq \psi$ on $\Gamma\times\{0\}$, find a $u$ such that

\begin{equation}\label{4.2.1b}
\begin{aligned}
\Delta u-\partial_tu= 0\ \ \  & \ \ \ {\rm{in}} \ \  \Omega \times (0,T] \\
\begin{rcases}  
&u\geq \psi, \ \ \alpha \partial_tu+\partial_\nu u\geq 0\ \ \ \\
&(\alpha \partial_tu+\partial_\nu)(u-\psi)=0\ \  \\
\end{rcases} &\ \ \ {\rm{on}} \ \  \Gamma \times (0,T] \\
u=\phi\ \ \  & \ \ \ {\rm{on}} \ \  \partial_p((\overline{\Omega}\setminus\Gamma) \times (0,T]) 
\end{aligned}
\end{equation}
where $\nu$ is the outward unit normal on $\partial\Omega$ and $\alpha$ any constant, $0<\alpha\leq 1$.

The approximate (penalized) problem is then

\begin{equation}\label{4-2penmodel2b}
\begin{cases}  
\partial_tu^\varepsilon-\Delta u^\varepsilon=0,   &{\rm{in}} \ \  \Omega \times (0,T] \cr
-\alpha\partial_tu^\varepsilon-\partial_\nu u^\varepsilon=\beta_\varepsilon(u^\varepsilon
-\psi^\varepsilon) &{\rm{on}} \ \  \Gamma \times (0,T] \cr
u^\varepsilon=\phi^\varepsilon+\varepsilon &{\rm{on}} \ \  \partial_p(\overline{\Omega}\setminus\Gamma \times (0,T]) \cr
\end{cases}
\end{equation}
where $\beta_\varepsilon$ is as in Section \ref{section3}. Let's extend $\psi^\varepsilon$ to all $\Omega$ i.e. we take any function $\tilde{\psi}^\varepsilon(x',x_n,t)$ such that $\tilde{\psi}^\varepsilon(x',0,t)=\psi^\varepsilon(x',t)$, $\tilde{\psi}^\varepsilon(x',x_n,t)<\phi$ on $\partial_p((\Omega\setminus\Gamma)\times(0,T])$ and $\frac{\partial\tilde{\psi}^\varepsilon}{\partial\nu}(x',0,t)=0.$ Subtracting $\tilde{\psi}^\varepsilon$ from the solution we have
\begin{equation}\label{4.2p2ab}
\begin{cases}  
\Delta (u^\varepsilon-\tilde{\psi}^\varepsilon)-\partial_t
(u^\varepsilon-\tilde{\psi}^\varepsilon)=
-(\Delta\tilde{\psi}^\varepsilon-
\partial_t\tilde{\psi}^\varepsilon) &{\rm{in}} \ \  \Omega \times (0,T] \cr
-\alpha\partial_t(u^\varepsilon
-\tilde{\psi}^\varepsilon)-\partial_\nu (u^\varepsilon-\tilde{\psi}^\varepsilon)
=\beta_\varepsilon(u^\varepsilon-\tilde{\psi}^\varepsilon)
+\alpha\partial_t\tilde{\psi}^\varepsilon &{\rm{on}} \ \  \Gamma \times (0,T] \cr
u^\varepsilon-\tilde{\psi}^\varepsilon
=\phi^\varepsilon-\tilde{\psi}^\varepsilon+\varepsilon &{\rm{on}} \ \  \partial_p(\overline{\Omega}\setminus\Gamma \times (0,T]). \cr
\end{cases}
\end{equation}
Differentiate with respect to time and set $v^\varepsilon=(u^\varepsilon-\tilde{\psi}^\varepsilon)_t$ to obtain
\begin{equation}\label{4.2p2bb}
\begin{cases}  
\Delta v^\varepsilon-\partial_tv^\varepsilon
=-(\Delta\tilde{\psi}^\varepsilon-
\partial_t\tilde{\psi}^\varepsilon)_t &{\rm{in}} \ \  \Omega \times (0,T] \cr
-\alpha\partial_tv^\varepsilon-\partial_\nu v^\varepsilon=\beta_\varepsilon'(u^\varepsilon
-\tilde{\psi}^\varepsilon)v^\varepsilon
+\alpha\partial_t\tilde{\psi}_t^\varepsilon &{\rm{on}} \ \  \Gamma \times (0,T] \cr
v^\varepsilon=(\phi^\varepsilon-\tilde{\psi}^\varepsilon)_t &{\rm{on}} \ \  \partial_p(\overline{\Omega}\setminus\Gamma \times (0,T]) \cr
v^\varepsilon=\Delta (\phi^\varepsilon-\tilde{\psi}^\varepsilon) &{\rm{on}} \ \  \Omega \times \{0\}. \cr
\end{cases}
\end{equation}

In order to avoid technicalities, in this more complex situation, and bring forth the main idea, we shall assume throughout this section that $(\Delta\tilde{\psi}^\varepsilon-
\partial_t\tilde{\psi}^\varepsilon)_t=0$ and we work with
\begin{equation}\label{4.2.2b}
\begin{cases}  
\Delta v^\varepsilon-\partial_tv^\varepsilon=0 &{\rm{in}} \ \  \Omega \times (0,T] \cr
-\alpha\partial_tv^\varepsilon-\partial_\nu v^\varepsilon=\beta_\varepsilon'(u^\varepsilon
-\tilde{\psi}^\varepsilon)v^\varepsilon
+\alpha\partial_t\tilde{\psi}_t^\varepsilon &{\rm{on}} \ \  \Gamma \times (0,T] \cr
v^\varepsilon=(\phi^\varepsilon-\tilde{\psi}^\varepsilon)_t &{\rm{on}} \ \  \partial_p(\overline{\Omega}\setminus\Gamma \times (0,T]) \cr
v^\varepsilon=\Delta (\phi^\varepsilon-\tilde{\psi}^\varepsilon) &{\rm{on}} \ \  \Omega \times \{0\}. \cr
\end{cases}
\end{equation}

We shall repeat the approach of Section \ref{subsection4.1} but, as it was done in \cite{AC2010}, 
instead of parabolic cylinders we take "hyperbolic" hypercubes with one of its sides lying on $\Gamma$. 
We normalize again i.e. our solution is between zero and one and we prove (Lemma \ref{lemma4.2.4b}) that, if $v^\varepsilon$ 
is zero on the top center and on $\Gamma$ in such a hypercube, then in a concentric subhypercube into the future $v^\varepsilon$ 
is smaller than one. Then we rescale and repeat. The rescaling, of course, is hyperbolic appropriate for the boundary term 
on $\Gamma$ but diminishes the time derivative in the heat equation; this though does not prevent us to obtain the continuity, as it was done in \cite{AC2010}.

Our first lemma asserts that if $v^\varepsilon$ is "most of the time" very near to its positive maximum in some hypercube sitting in $(\R^{n+1})$ against the hyperplane $x_n=0$ and going backwards in time then in a smaller hypercube into the future, $v^\varepsilon$ is strictly positive.

\begin{lemma}\label{lemma4.2.1b}
Let $Q_R(x_0',0,t_0)\subset \Omega\times(0,T]$ 
where $Q_R(x_0',0,t_0)=B_R(x_0',0)\times(t_0-R,t_0]$, $B_R(x_0',0)=B_R'(x_0)\times(0,1)$, $B'_R(x_0)=\{x'=(x_1,...,x_{n-1}):|x_i'-x_{0i}'|<R,i=1,...,n-1\}$ and $Q_R(x_0',0)=B_R'(x_0')\times(t_0-R,t_0]$. Suppose that $0<v^\varepsilon<1$ in $Q_R(x_0,t_0)$ 
where $v^\varepsilon$ is a solution to (\ref{4.2.2b}). Then there exists a constant $\sigma>0$,
 independent of $\varepsilon$, such that $$\fint_{Q'_R(x'_0,t_0)}(1-v^\varepsilon)^2dx'dt+\fint_{Q_R(x_0,t_0)}(1-v^\varepsilon)^2dxdt<\sigma$$
implies that $$v^\varepsilon\geq \frac{1}{8}$$ in $Q_{r/8}(x_0,t_0)=B_R'(x_0')\times(0,\frac{1}{8})\times(t_0-\frac{R}{8},t_0]$.
\end{lemma}
\begin{proof}
For simplicity we drop the superscript $\varepsilon$, shift $(x_0,0,t_0)$ to $(0,0,0)$ and write $Q_R$ for $Q_R(0,0,0)$. We first derive an energy inequality associated to our problem. Set $w=1-v$ then the problem becomes
\begin{equation}\label{w4.2.2b}
\begin{cases}  
\Delta w-\partial_tw=0 &{\rm{in}} \ \  \Omega \times (-T,T] \cr
\alpha\partial_tw+\partial_\nu w=\beta_t(u-\tilde{\psi})+\alpha\partial_t\tilde{\psi}_t &{\rm{on}} \ \  \Gamma \times (-T,T] \cr
w=1-(\phi-\tilde{\psi})_t &{\rm{on}} \ \  \partial_p(\overline{\Omega}\setminus\Gamma \times (-T,T]) \cr
w=1-\Delta (\phi-\tilde{\psi}) &{\rm{on}} \ \  \Omega \times \{-T\}. \cr
\end{cases}
\end{equation}
Choose a smooth cut-off function $\zeta$ vanishing near the parabolic boundary of $Q_R$ except on $Q_R'$ and $k\geq 0$. Multiply the above by $\zeta^2(w-k)^+$ and integrate by parts to obtain
$$\int_{Q_R}[\grad(\zeta^2(w-k)^+)\grad w+\zeta^2(w-k)^+\partial_tw]dxdt-\int_{Q_R'}\zeta^2(w-k)^+\partial_\nu wdx'dt=0$$
and
$$\int_{Q_R}[|\grad(\zeta^2(w-k)^+)|^2+\frac{1}{2}\partial_t(\zeta^2(w-k)^+)^2]dxdt-\int_{Q_R'}\zeta^2(w-k)^+(\beta_t(u-\psi)+\alpha\psi_{tt}
-\alpha\partial_tw)dx'dt=$$
$$=\int_{Q_R}((w-k)^+)^2(|\grad \zeta|^2+\zeta\zeta_t)dxdt$$
and
\begin{eqnarray}\label{p5-1b}
\frac{\alpha}{2}\int_{Q_R'}\partial_t[(\zeta(w-k)^+)^2]dx'dt&+& \int_{Q_R}\bigg[\frac{1}{2}\partial_t[(\zeta(w-k)^+)^2]+|\grad(\zeta(w-k)^+)|^2\bigg]dxdt\nonumber\\ 
&=&\int_{Q_R'}\partial_t[\zeta^2(w-k)^+\beta(u-\psi)]dx'dt-\int_{Q_R'}\zeta^2\partial_t(w-k)^+\beta(u-\psi)dx'dt\nonumber\\ 
&+&2\int_{Q_R'}\zeta\partial_t\zeta(w-k)^+\beta(u-\psi)dx'dt+\alpha\int_{Q_R'}\zeta\partial_t
\zeta((w-k)^+)^2dx'dt\nonumber\\ 
&+&\alpha\int_{Q'_R}\zeta^2(w-k)^+\psi_{tt}dx'dt+\int_{Q_R}(|\grad \zeta|^2+\zeta\partial_t\zeta)[(w-k)^+]^2dxdt.\nonumber\\
\end{eqnarray}
Now, using the fact that $\beta$ is bounded and negative, $(u-\psi)_{tt}$ is bounded below and since the upper limit of the $t-$integration $t=0$ can be replaced by any $-R\leq t\leq 0$, we obtain
$$\frac{\alpha}{2}\max_{-R\leq t\leq 0}\int_{B_R'}[(w-k)^+\zeta]^2dx'+\frac{1}{2}\max_{-R\leq t\leq 0}\int_{B_R}[(w-k)^+\zeta]^2dx+\int_{Q_R}|\grad ((w-k)^+\zeta)|^2dxdt $$
$$\leq ||\beta||_{\infty}||(u-\psi)_{tt}^-||_{\infty}\int_{Q_R\cap\{w>k\}}\zeta^2dx'dt+2||\beta||_{\infty}\int_{Q_R'}(w-k)^+|\partial_t\zeta|dx'dt$$
$$+\alpha\int_{Q_R'}[(w-k)^+]^2|\partial_t\zeta|dx'dt+\alpha||\psi_{tt}||_{\infty}\int_{Q_R'}(w-k)^+dx'dt+\int_{Q_R}[(w-k)^+]^2(|\grad\zeta|^2+|\partial_t\zeta|)dxdt$$
or
$$\alpha\max_{-R\leq t\leq 0}\int_{B_R'}[(w-k)^+\zeta]^2dx'+\max_{-R\leq t\leq 0}\int_{B_R}[(w-k)^+\zeta]^2dx+\int_{Q_R}|\grad ((w-k)^+\zeta)|^2dxdt$$
$$\leq C\bigg(\int_{Q_R'}[(w-k)^+]^2|\partial_t\zeta|+(w-k)^+(1+|\partial_t\zeta|)dx'dt+\int_{Q_R'\cap\{w>k\}}
\zeta^2dx'dt\bigg)$$
$$+C\int_{Q_R}[(w-k)^+]^2(|\grad\zeta|^2+|\partial_t\zeta|)dxdt$$
where $C=2\max\{||\beta||_{\infty},||(u-\psi)_{tt}^-||_{\infty},1,\alpha||\psi_{tt}||_{\infty}\}$ and, a fortiori, we have the "energy inequality" 
$$\alpha\max_{-R\leq t\leq 0}\int_{B_R'}[(w-k)^+\zeta]^2dx'+\int_{Q_R}|\grad ((w-k)^+\zeta)|^2dxdt$$
\begin{equation}\label{eq_4.21b}
\leq C\bigg(\int_{Q_R'}[(w-k)^+]^2|\partial_t\zeta|+(w-k)^+(1+|\partial_t\zeta|)+\zeta^2\chi_{\{w>k\}}dx'dt+\int_{Q_R}[(w-k)^+]^2(|\grad\zeta|^2+|\partial_t\zeta|)dxdt\bigg).
\end{equation}

Now that we have our energy inequality we shall obtain an iterative sequence of inequalities.
 More precisely, the method consists in taking a sequence of decreasing cutoffs in space and 
time $\zeta_m$ that converges to the indicator function of $Q_{R/4}$ and simultaneously a series 
of cutoffs of the graph of $u$, $u_m$ that converge to $(w-7/8)^+$ and prove by iteration that in 
the limit $(w-7/8)^+=0$ on $Q_{r/4}$. We follow closely the corresponding argument in \cite{AC2010}. We, therefore, define 
$$k_m=\frac{9}{16}+\frac{1}{16}\bigg(1-2^{-m}\bigg)\ \ \ R_m=\frac{R}{4}\bigg(1+\frac{1}{2^m}\bigg)$$
$$Q_m'=\{(x_1,...,x_n,t):-R_m\leq x_i\leq R_m, -R_m\leq t\leq 0\}$$
and we choose the cutoff functions $\zeta_m$ to depend only on $x'$ and $t$ such that $\chi_{Q_{m+1}'}\leq \zeta_m\leq \chi_{Q_{m}'}$, $|\grad\zeta_m|\leq C2^m$ and $|\partial_t \zeta_m|\leq C2^m$. We set $u_m=(u-k_m)^+$ and we denote by
$$I_m=\int_{Q'_R}(\zeta_mu_m)^2dx'dt+\int_0^{\delta^m/2}\int|\grad(\zeta_mu_m)|^2dxdt$$
where $0<\delta<1$ is chosen such that $2^n2^{-\frac{(n+6)2^{-m-1}}{\delta^m}}\leq 2^{-m-4}$ holds.
We also choose $M$ as in \cite{AC2010} to satisfy $2^{n+1}M^{-\frac{m}{2}}(\delta^n)^{-m-1}\leq 2^{-m-6}$, $M^{-m}\geq C4^{m(1+\frac{1}{n-1})}M^{-(m-3)(1+\frac{1}{n-1})}$, $m\geq 14(n-1)$.

We want to prove simultaneously that for every $m\geq 0$, $I_m\leq M^{-m}$ and that $u_m=0$ on $Q_{m}'\times \{\frac{\delta^m}{2}\}$.  The proof is by induction and is identical with Step 2a and Step 2b of Lemma 2.2 in \cite{AC2010} except that 
$$||u_{\chi_{Q_R'}}\ast H(x_n)||\leq ||H(y)||_{\infty(\{x_n\geq 1\})}\int_{Q_R'}udx'dt\leq \frac{2^{n+2}}{\pi^{\frac{n}{2}}}\bigg(\frac{n+2}{2e}\bigg)^{n+2}|Q_R'|^{1/2}\sigma^{1/2}\leq \frac{1}{64}$$
for $\sigma$ small enough. So we concentrate on Step 2c, where we will show that 
$$I_m\leq C4^{m(1+\frac{1}{n-1})}I_{m-3}^{1+\frac{1}{n-1}}, \ \ m\geq 14n-13.$$
By the energy inequality,
$$I_m\leq \int(w_m\zeta_{m-1})^2dx'dt
+C\bigg[C2^m\int(w_m\zeta_{m-1})^2dx'dt+(1+C2^m)\frac{1}{2}\int(w_m\zeta_{m-1})^2dx'dt\bigg]$$ $$+C\bigg[(1+C2^m)\frac{1}{2}|Q_{m-1}\cap\{w_m\neq 0\}|+|Q'_{m-1}\cap\{w_m\neq 0\}|+(C2^m)^2\int(w_m\zeta_{m-1})^2dxdt\bigg]$$
$$+C\bigg[\frac{1}{2}\int(w_m\zeta_{m-1})^2dxdt+\frac{1}{2}|Q_{m-1}\cap\{w_m\neq 0\}|\bigg]$$
where we have used Young's inequality. Since $w_m<w_{m-1}$ and $\{w_m\neq 0\}=\{w_{m-1}>2^{-m-4}\}$, we have 
$$I_m\leq C2^m\int(w_{m-1}\zeta_{m-1})^2dx'dt+C4^m\int(w_{m-1}
\zeta_{m-1})^2dxdt.$$
Also, the integral of the second term i.e
$$\int(w_{m-1}\zeta_{m-1})^2dxdt\leq \int|(w_{m-2}\zeta_{m-2})\ast H(x_n)|^2dxdt\leq ||H||^2_{L^1(Q_R)}\int(w_{m-2}\zeta_{m-2})^2dx'dt.$$
Therefore
\begin{eqnarray}
I_m &\leq & C4^m\int(w_{m-2}\zeta_{m-2})^2dx'dt\nonumber\\
&\leq &  C4^m\bigg(\int(w_{m-2}\zeta_{m-2})^{2\frac{n}{n-1}}dx'dt\bigg)^{\frac{n-1}{n}}|\{w_{m-2}\neq 0\}\cap Q_{m-2}'|^{\frac{1}{n}}\nonumber\\
&\leq & C4^{m(1+\frac{1}{n-1})}\int(w_{m-3}\zeta_{m-3})^{2\frac{n}{n-1}}dx'dt.
\end{eqnarray}
By Sobolev's inequality
$$I_m\leq C4^{m(1+\frac{1}{n-1})}\bigg(\int(w_{m-3}\zeta_{m-3})^2dx'dt+\int|\Lambda^{1/2}(w_{m-3}\zeta_{m-3})|^2dx'dt\bigg)^{\frac{n}{n-1}}$$
where $\Lambda(w_{m-3}\zeta_{m-3})=-\frac{\partial}{\partial x_n}(w_{m-3}\zeta_{m-3})$. Since 
$$\int|\Lambda^{1/2}(w_{m-3}\zeta_{m-3})|^2dx'dt\leq \int|\grad(w_{m-3}\zeta_{m-3})|^2dxdt$$
we have $$I_m\leq C4^{m(1+\frac{1}{n-1})}I_{m-3}^{1+\frac{1}{n-1}}\ \ {\rm{for\ every}}\ \ m\geq 14(n-1)+1$$
i.e. $I_m\rightarrow 0$ as $m\rightarrow \infty$ provided that
$$I_0\leq C^{-(n-1)}4^{-n(n-1)}=:\sigma.$$

Hence to complete the proof, consider the function $\bar{w}$ defined by 

\begin{equation}\label{barwb}
\begin{cases}  
\Delta \bar{w}-\partial_t\bar{w}=0 &{\rm{in}} \ \  Q_{R/4} \cr
\bar{w}=1 &{\rm{on}} \ \  \partial_p(\overline{Q_{R/4}}\setminus\{x_n=0\} \cr
\bar{w}=\frac{5}{8} &{\rm{on}} \ \  Q_{R/4}'. \cr
\end{cases}
\end{equation}
Then $\bar{w}<7/8$ in $Q_{R/8}$ and by the maximum principle $w\leq \bar{w}$. 
\end{proof}

Our second lemma asserts that if $v^\varepsilon$ is very tiny "most of the time" in some hypercube (as above) then, in a smaller concentric hypercube, $v^\varepsilon$ goes down from $1$ to $7/8$.

\begin{lemma}\label{lemma4.2.2b}
Let $Q_R(x_0',0,t_0)$ be as in Lemma \ref{lemma4.2.1b}. Suppose that $v^\varepsilon$ is a subsolution to (\ref{4.2.2b}) and that $0<v^\varepsilon<1$ in $Q_R(x_0',0,t_0)$. Then there exists a constant $\bar{\sigma}>0$, independent of $\varepsilon$, such that 
$$\fint_{Q'_R(x'_0,t_0)}(v^\varepsilon)^2dx'dt+\fint_{Q_R(x_0',0,t_0)}(v^\varepsilon)^2dxdt<\bar{\sigma}$$
implies that $$v^\varepsilon\leq \frac{7}{8}$$ in $Q_{r/8}(x_0,0,t_0)$.
\end{lemma}

\begin{proof}
The proof is identical to the proof of Lemma \ref{lemma4.2.1b} except from the energy inequality. For simplicity again we drop the $\varepsilon$ and take $(x_0',0,t_0)=(0,0,0)$ with $Q_R=Q_R(0,0,0)$. Since $\beta'\geq 0$, $v$ satisfies 
\begin{equation}\label{4.2.2bb}
\begin{cases}  
\Delta v-\partial_tv=0 &{\rm{in}} \ \  Q_R \cr
-\alpha\partial_tv-\partial_\nu v\geq \alpha\partial_t\psi_t &{\rm{on}} \ \  Q_R' \cr
\end{cases}
\end{equation}
Choose again, a smooth cutoff function $\zeta$ vanishing near the parabolic boundary of $Q_R$ except on $Q_R'$ and $k\geq 0$. Multiply the above by $\zeta^2(v-k)^+$ and integrate by parts to obtain
$$\int_{Q_R}[\grad(\zeta^2(v-k)^+)\grad v+\zeta^2(v-k)^+\partial_tv]dxdt-\int_{Q_R'}\zeta^2(v-k)^+\partial_\nu vdx'dt\leq 0$$
or 
$$\int_{Q_R}\bigg[|\grad(\zeta(v-k)^+)|^2+\frac{1}{2}\partial_t((\zeta(v-k)^+)^2)\bigg]dxdt+\int_{Q_R'}\zeta^2(v-k)^+\bigg[\alpha\psi_{tt}+\alpha\partial_tv\bigg]dx'dt$$ $$\leq \int_{Q_R}((v-k)^+)^2\bigg[|\grad \zeta|^2+\zeta\partial_t\zeta\bigg]dxdt$$
and
$$\frac{\alpha}{2}\int_{Q_R'}\partial_t[(\zeta(v-k)^+)^2]dx'dt+\int_{Q_R}\bigg[\partial_t[(\zeta(v-k)^+]^2)+|\grad(\zeta(v-k)^+)|^2\bigg]dxdt$$
$$\leq \alpha\int_{Q_R'}[(v-k)^+]^2|\partial_t\zeta|dx'dt+\alpha||\psi_{tt}||_{L^{\infty}}\int_{Q_R'}(v-k)^+dx'dt+\int_{Q_R}[(v-k)^+]^2(|\grad \zeta|^2+|\partial_t\zeta|)dxdt.$$
And again taking as upper limit any $-R\leq t\leq 0$ we obtain
$$\alpha\max_{-R\leq t\leq 0}\int_{B_R'}[\zeta(v-k)^+]^2dx'+\int_{Q_R}|\grad(\zeta(v-k)^+)|^2dxdt$$
$$\leq \bar{C}\bigg\{\int_{Q_R'}\bigg[[(v-k)^+]^2|\partial_t\zeta|+(v-k)^+\bigg]dx'dt+\int_{Q_R}[(v-k)^+]^2(|\grad \zeta|^2+|\partial_t\zeta|)dxdt\bigg\}$$
where $\bar{C}=2\max\{1,\alpha||\psi_{tt}||_{L^{\infty}}\}$.

Now, since we have our energy inequality, the rest is as that of Lemma \ref{lemma4.2.1b} and we define $$\bar{\sigma}:=\bar{C}^{-(n-1)}4^{-n(n-1)}.$$
\end{proof}

We proceed, now, by using the parabolic version of DeGiorgi's isoperimetric lemma. This lemma is proved in \cite{CaffVass} and with proper adjustments applies to our situation. We state it as our next lemma.

\begin{lemma}\label{lemma4.2.3b}
Given $\epsilon_1>0$, there exists a $\delta_1>0$ such that for every subsolution $v^\varepsilon$ to (\ref{4.2.2b}) satisfying $0<v^\varepsilon<1$ in $Q_R$, $$|\{(x,t)\in Q_R:v^\varepsilon=0\}|\geq \sigma_0|Q_R|$$ if $$|\{(x,t)\in Q_R:0<v^\varepsilon<1/2\}<\delta_1|Q_R|$$ 
then 
$$\fint_{Q'_{R'}}\bigg[(v^\varepsilon-\frac{1}{2})^+\bigg]^2dx'dt+\fint_{Q_{R'}}\bigg[(v^\varepsilon-\frac{1}{2})^+\bigg]^2dxdt<C\sqrt{\epsilon_1}$$
where $R'=\frac{\sigma_0}{2}R$ for $\sigma_0>0$.
\end{lemma}

We are now ready to obtain our basic decay estimate to zero.

\begin{lemma}\label{lemma4.2.4b}
Let $Q_R(x_0',0,t_0)$ and $\sigma$ be as in Lemma \ref{lemma4.2.1b}. Suppose that $v^\varepsilon$ is a solution to (\ref{4.2.2b}) such that $v^\varepsilon(x_0',0,t_0)=0$ and $0 \leq v^\varepsilon\leq 1$ in $Q_{R}(x_0',0,t_0)$. Then $v^\varepsilon\leq 1-C\sigma$ in $Q_{R'}(x_0',0,t_0)$ where $C$ is independent of $\varepsilon$ and $R'=\frac{\sigma}{16}R$.
\end{lemma}

\begin{proof}
We drop the $\varepsilon$, take $(x_0',0,t_0)$ to be $(0,0,0)$ (by translation), and set $Q_R=Q_R(0,0,0)$. Since $v(0,0,0)=0$ by Lemma \ref{lemma4.2.1b}
$$\fint_{Q'_R}(1-v)^2dx'dt+\fint_{Q_R}(1-v)^2dxdt\geq \sigma.$$
It follows then that there exists a constant $c_0<1$ such that 
$$|\{v<1-\frac{\sigma}{4}\}\cap Q_R|\geq c_0\sigma|Q_R|.$$
Therefore set $$w:=\frac{4}{\sigma}\bigg(v-(1-\frac{\sigma}{4})\bigg)^+$$
and observe that $w$ is a (nonnegative) subsolution to (\ref{4.2.2b}). By DeGiorgi again, the normalized truncations i.e. 
$$w_k:=2^k\bigg(w-(1-2^{-k})\bigg)^+$$
are still subsolutions to (\ref{4.2.2b}). We will show, now that in a finite number of 
steps $k_0=k_0(\delta_1)$ ($\delta_1$ as in Lemma \ref{lemma4.2.3b}) that $|\{w_{k_0}>0\}|=0$. 
Note that for every $k$, $0\leq w_k\leq 1$ and $|\{w_k=0\}\cap Q_R|\geq \sigma_1|Q_R|$. 
Set $C\sqrt{\epsilon_1}\leq \bar{\sigma}$ where $\epsilon_1$ is defined in Lemma \ref{lemma4.2.3b} and $\bar{\sigma}$ in 
Lemma \ref{lemma4.2.2b}. Hence we assume that for every $k$, $|\{0<w_k<\frac{1}{2}\}\cap Q_R|\geq \delta_1|Q_R|$. Then for every $k$
$$|\{w_k=0\}|=|\{w_{k-1}=0\}|+|\{0<w_{k-1}<1/2\}|\geq |\{w_{k-1}=0\}|+\delta_1|Q_R|$$
Hence after a finite number of steps say $k_0>1/\delta_1$ we get $|\{w_{k_0}=0\}|\geq |Q_R|$. Thus $w_{k_0}<0$ i.e. $2^{k_0}[w-(1-2^{-k_0})]^+=0$ or $w<1-2^{-k_0}$. 
Suppose, now, that there exists $k'$, $0\leq k'\leq k_0$ such that $$|\{0<w_{k'}<\frac{1}{2}\}|<\delta_1.$$
By Lemma \ref{lemma4.2.3b} applied to $w_{k'}$ and consequently by Lemma \ref{lemma4.2.2b} applied to $w_{k'+1}$ we conclude that $w_{k'+1}\leq 7/8$ in $Q_{R'}$, where $R'=\frac{\sigma}{16}R$ 
i.e. $w<1-\frac{1}{8}2^{-(k'+1)}$. A fortiori, in both cases we have $w<1-2^{-(k_0+4)}$ in $Q_{R'}$ that is $v<1-2^{-k_0-6}\sigma$.
\end{proof}

The estimates we obtained above are all independent of $\varepsilon$ and remain invariant under hyperbolic scaling much the same way as in \cite{AC2010}. 
Although the time derivative term diminishes in the rescaling, we still obtain the continuity of the time derivative.

\begin{prop}\label{Prop4.2.5b}
Let $v^\varepsilon$ be a solution to (\ref{4.2.2b}) in $Q_R$. Suppose that $0\leq v^\varepsilon\leq M$ where $M$ is independent of $\varepsilon$. 
If $v^\varepsilon(0,0,0)=0$ then 
$$v^\varepsilon(x',x_n,t)\leq \omega(|x'|,|x_n|,|t|)$$
where $\omega$ is a modulus of continuity (i.e. monotone and $\omega(0)=0$) independent of $\varepsilon$.
\end{prop}
\begin{proof}
We drop as usual the $\varepsilon$. Set 
$$Q_k:=Q_{R_k}=(-R_k,R_k)^{n-1}\times (0,r_k)\times (-R_k,0] \ \ {\rm{and}}\ \ M_k:=\sup_{Q_k}v$$
where $R_k:=r_kR$, $r_k:=\frac{\sigma}{16}M_k$.
Define 
$$\bar{v}:=\frac{v_k}{M_k}, \ \ {\rm{where}}\ \ v_k(x,t):=v(r_kx', r_kx_n,r_kt).$$
Then $\bar{v}$ verifies 
\begin{equation}\label{page17Prob}
\begin{cases}  
\Delta \bar{v}-r_k\partial_t\bar{v}=0 &{\rm{in}} \ \  Q_R \cr
-\alpha\partial_t\bar{v}-\partial_\nu \bar{v}=\beta'(u-\psi)\bar{v}+\bar{\psi}_t &{\rm{on}} \ \  Q_R' \cr
\end{cases}
\end{equation}
where $\bar{\psi}_t=\alpha\partial_t\psi_t/M_k$. We apply Lemma \ref{lemma4.2.4b} to $\bar{v}$ to obtain
$$\sup_{Q_{R'}}\bar{v}\leq 1-C\sigma.$$

Hence in our original setting
$$\sup_{Q_{k+1}}v\leq \mu_k\sup_{Q_k}v$$
where $\mu_k=1-C(\sup_{Q_k}v)^{n-1}$. Therefore $\mu_k\rightarrow 1$ as $k\rightarrow \infty$ only if $\sup_{Q_k}u\rightarrow 0$ which yields a modulus of continuity. Finally, a standard barrier argument yields the continuity from the future, too.
\end{proof}

\begin{thm}\label{theorem4.2.6b}
Let $u$ be a solution to (\ref{4.2.1b}) then $(u-\psi)_t^+$ is continuous with a uniform modulus of continuity. 
\end{thm}





\section{Further implications on the (nondynamic) thin obstacle problem or (time dependent) Signorini Problem}\label{section5}

In the present section we shall concentrate on the nondynamic parabolic "thin" obstacle or parabolic Signorini problem and we will show how the quasi-convexity yields the optimal regularity of the solution as well as free boundary regularity. The other cases, as it was mentioned in \S\ref{Intro}, will be treated in forthcoming papers. Since it is easier to work with the zero obstacle, we extend the obstacle as it was done in \S\ref{subsection4.2} in all of $\Omega$ and subtract it from the solution which we still denote by $u$. More precisely:

Given $\Omega \subset \R^n$ be an open bounded set with smooth
boundary $\partial \Omega$ and $\Gamma \subset \partial \Omega$ lying in $\R^{n-1}$.
We consider the following problem:
\begin{equation}\label{model2}
\begin{cases}  
\Delta u-\partial_t u=f,   &{\rm{in}} \ \  \Omega \times (-T,T] \cr
\partial_\nu u\geq 0, \ \ u\geq 0  &{\rm{on}} \ \  \Gamma \times (-T,T] \cr
u\partial_\nu u=0  &{\rm{on}} \ \  \Gamma \times (-T,T] \cr
u=\phi-\psi &{\rm{on}} \ \  \partial_p(\Omega\setminus\Gamma \times (-T,T]) \cr
\end{cases}
\end{equation}
where $\nu$ is the unit outward normal, the functions $\psi (x',t)$ and $\phi (x,t)$ are smooth functions, satisfying
the compatibility conditions of \S\ref{section3}, and $f:=-(\Delta \tilde{\psi}-\partial_t \tilde{\psi})$. Notice that the extended $\tilde{\psi}$ can be chosen, with no loss of generality, in such a way so that $f$ is independent of $x_n$.

The methods to follow can be easily extended to cover a more general nonhomogeneous term $f$. But, in order to avoid minor technicalities and set forth the ideas involved behind it, we work with (\ref{model2}).

\subsection{Optimal regularity of the space derivative}\label{optimal space}






The solution to the problem (\ref{model2}) is globally Lipschitz continuous in space and furthermore the space normal to the hyperplane derivative 
enjoys a $C^{\alpha}$ for $0<\alpha\leq{\frac{1}{2}}$ parabolic regularity up to the hyperplane (see \cite{Athparabolic} and \cite{ArU}). We will prove in this subsection that, actually, $\alpha=\frac{1}{2}$. Recently, in \cite{DGPT}, the optimal space derivative regularity was also obtained using the parabolic Almqren's frequency formula approach.

First, we want to complete what had started in \cite{AC} i.e. to prove a parabolic monotonicity formula analogous to the elliptic one for the global zero obstacle case. We thus take in (\ref{model2}) $f=0$ and the domain $\Omega$ to be the half space $\R^n_+$. In this situation, it is clear, perhaps by appropriately blowing up the local solution, that the solution $u$ is convex in the tangential and time directions. For simplicity we take the origin to be a free boundary point. The proof of the monotonicity result relies on the following eigenvalue problem (see the appendix of \cite{AC}): 

\begin{lemma}\label{Lemma3}
Set $$\lambda_0=\inf_{\substack{w\in H^1(\R^n_+) \\ w=0\ \text{on}\ \R^{n-1}_-}}\frac{\int_{\R^n_{+}}|\grad w(y,-1)|^2e^{-\frac{-|y|^4}{4}}dy}{\int_{\R^n_{+}}w^2(y,-1)e^{-\frac{-|y|^4}{4}}dy},$$
where $$\R^n_+:=\{x=(x',x_n)\in \R^n : x_n>0\}$$
and 
$$\R^{n-1}_-:=\{(x',0):x'\in \R^{n-1},\ x_{n-1}<0\}.$$
Then $\lambda_0=1/4$.
\end{lemma}
Let $w$ be any  function in $\overline{\R^n_+}\times [-1,0]$ that is caloric in $\R^n_+\times [-1,0]$, where $\R^n_+=\{x=(x',x_n)\in \R^n : x_n>0\}$. We assume that $w$ has moderate growth at infinity, 
$$\int_{B_R}w^2(x,-1)dx\leq Ce^{\frac{|x|^2}{4+\varepsilon}}$$
for some positive constant $C$, $R$ large and some $\varepsilon>0$. We also set 
\begin{equation}\label{fundamental solution}G(x,t)= \begin{cases}\frac{1}{(4\pi t)^{n/2}}e^{-\frac{|x|^2}{4t}}, &  t> 0\cr
 0 &  t\leq 0.
\end{cases}\end{equation}

\begin{lemma}\label{Lemma4}

Set $w(x,t)=u_{x_n}(x,t)$ where $u$ is a solution, with the above restrictions, to problem (\ref{model2}) and assume that $w(0,0)=0$. If $$\varphi(t):=\frac{1}{t^{1/2}}\int_{-t}^0\int_{\R^n_+}|\grad w|^2G(x,-s)dxds,$$ then $\varphi(t)$ is increasing in $t$.
\end{lemma}
\begin{proof}
Note that $\Delta w^2=2w\Delta w+2|\grad w|^2$. We compute $\varphi'(t)$, with a usual mollification argument, to obtain
\begin{eqnarray}
\varphi'(t) &=& -\frac{1}{2t^{3/2}}\int_{-t}^0\int_{\R^n_+}|\grad w|^2G(x,-s)dxds+\frac{1}{t^{1/2}}\int_{\R^n_+}|\grad w(x,-t)|^2G(x,t)dx\nonumber\\
&=& -\frac{1}{2t^{3/2}}\int_{-t}^0\int_{\R^n_+}(\frac{1}{2}\Delta w^2-ww_t)G(x,-s)dxds+\frac{1}{t^{1/2}}\int_{\R^n_+}|\grad w(x,-t)|^2G(x,t)dx.\nonumber
\end{eqnarray}
By integrating by parts and noticing that $\Delta G+G_t=\delta_{(0,0)}$, $w(0,0)=0$ and $G(x,0)=0$, we obtain, 


$$\varphi'(t)=-\frac{1}{4t^{3/2}}\int_{\R^n_+} w^2(x,-t) G(x,t)dx+\frac{1}{t^{1/2}}\int_{\R^n_+}|\grad w(x,-t)|^2G(x,t)dx$$
\begin{equation}
-\frac{1}{4t^{3/2}}\int_{-t}^0\int_{\R^{n-1}}2ww_\nu G(x',0,-s)dx'ds.
\end{equation}
Hence, by the eigenvalue problem of Lemma \ref{Lemma3} and the complimentary conditions of the solution on $R^{n-1}$, $\varphi'(t)\geq0$.
\end{proof}

\begin{thm}\label{Theorem4}
If $u$ is a solution to the global convex case of (\ref{model2}) then $\grad u\in C^{1/2,1/4}_{x,t}$ up to the coincidence set.
\end{thm}
\begin{proof}
It is enough to prove that $u$ tends to zero in a parabolic $C^1$ fashion as $(x,t)$, a point in the noncoincidence set, approaches a point $(x_0,t_0)$ in the coincidence set which we take to be the origin. Set $w=u_{x_n}$, then $w$ satisfies the hypothesis of Lemma \ref{Lemma4}. In particular, $w$ vanishes at the origin therefore
\begin{equation}\label{star1}
\frac{1}{t^{1/2}}\int_{-t}^0\int_{\R^n_+}|\grad w(x,s)|^2G(x,-s)dxds\leq C.
\end{equation}
Since $w$ vanishes on at least half of the space for all $t\leq 0$, the Poincare inequality implies that
\begin{equation}\label{star2}
\int_{\R^n_+}w^2(y,-r^2)G(x-y,t+r^2)dy\leq 4r^2\int_{\R^n_+}|\grad w(y,-r^2)|^2G(x-y,t+r^2)dy.
\end{equation}
Since $w^2$ is a subsolution across $x_n=0$ we have, for every $(x,t)\in Q_{r/2}^-$ and $s<r/2$
\begin{equation}\label{star3}
w^2(x,t)\leq \int_{\R^n}w^2(y,s)G(x-y,t-s)dy. 
\end{equation}
Now integrate (\ref{star3}) with respect to $s$ from $-r^2$ to $-r^2/2$ to obtain
\begin{equation}\label{star4}
r^2w^2(x,t)\leq \int_{-r^2}^{-r^2/2}\int_{\R^n}w^2(y,s)G(x-y,t-s)dyds
\end{equation}
and combining with Poincare inequality we have
\begin{equation}\label{star5}
w^2(x,t)\leq 4\int_{-r^2}^{-r^2/2}\int_{\R^n}|\grad w|^2G(x-y,t-s)dyds
\end{equation}
for every $(x,t)\in Q_{r/2}^-$. Hence by (\ref{star1}), the proof is complete.
\end{proof}
Now, we remove the restrictions previously imposed and we show how to improve the $0<\alpha<1$ in the $C^\alpha$ regularity to get $C^{1/2}$. First we prove a lemma, which uses the normal semi-concavity, the tangential semi-convexity, and the time semi-convexity. 
\begin{lemma}\label{Lemma5}
Let $u$ be a solution of (\ref{model2}) in $Q^+_1$ with $\grad u,u_t^+ \in C_{x,t}^{\alpha,\frac{\alpha}{2}}$ Then there exists a $\delta=\delta(\alpha)>0$ such that $$(0,0,t)\notin \Gamma(\{u_{x_n}<-r^{\alpha+\delta}\}\cap Q_r')$$
for every $t\in[-r^2,0]$ and $0<r<1$, where $\Gamma(A)$ denotes the convex hull of the set $A$.
\end{lemma} 
\begin{proof}
If $$(x',0,-r^2)\in \{u_{x_n}<-r^{\alpha+\delta}\}$$ then $$u(x',h,-r^2)\leq -r^{\alpha+\delta}h+\frac{M}{2}h^2$$ since $u_{x_nx_n}<M$. Take $h=\frac{r^{\alpha+m\delta}}{M}$ for some $m>1$; in this case $$u(x',h,-r^2)\leq -\frac{r^{2\alpha+(m+1)\delta}}{2M}.$$ Moreover, if we restrict the considerations to $|x'|\leq \frac{r}{2M}$ then 
\begin{equation}\label{onesidedelta}
u(x',h,-r^2)+M|x'|^2\leq -\frac{r^{2\alpha+(m+1)\delta}}{4M}
\end{equation}
provided that $\delta<\frac{2(1-\alpha)}{m+1}$. On the other hand, since $u_{tt}>-M_1$ and $u_t^+$ is Holder continuous whose exponent, with no loss of generality, can be taken to be the same $\alpha$ as above, we have
\begin{eqnarray}\label{twosidedelta}
u(0,h,-r^2)&\geq &u(0,h,0)-max\{0,c_1 h^{\alpha}r^2\}-\frac{M_1}{2}r^4\nonumber\\
&\geq &-c_0h^{1+\alpha}-max\{0,c_1 h^{\alpha}r^2\}-\frac{M_1}{2}r^4\nonumber\\
&>&-\bar{c}r^{(\alpha+m\delta)(1+\alpha)}
\end{eqnarray}
 
Finally, if we choose $\delta>\frac{\alpha(1-\alpha)}{\alpha m-1}$ and $m>1+\frac{2}{\alpha}$
we get a contradiction to (\ref{onesidedelta}) above. Note that the same argument applies for any $t\in[-r^2,0]$.
\end{proof}

We provide now our monotonicity formula for solutions to the local situation. 

\begin{lemma}\label{Lemma6non-a}
Let $\delta>0$ and $u$ be a solution to the Signorini problem (\ref{model2}). Set $w=u_{x_n}$ and $$\varphi(r)=\frac{1}{r}\int_{-r^2}^0\int_{\R^n_+}|\grad(\eta w)(x,s)|^2G(x,-s)dxds$$
for $r<1$ where $\eta\in C_0^\infty(B_r)$ with $\eta\equiv 1$ and $\eta_{x_n}|_{B_r\cap\R^{n-1}}=0$. There exists a universal constant $C>0$ such that
\begin{enumerate}
\item[(i)] if $2\alpha+\delta>1$ then $\varphi(r)\leq C$,
\item[(ii)] if $2\alpha+\delta<1$ then $\varphi(r)\leq Cr^{2\alpha+\delta-1}$.
\end{enumerate}
\end{lemma}
\begin{proof}
We compute
\begin{equation}\label{inter6.1a}
|\grad(\eta w)|^2=\frac{1}{2}(\Delta (\eta w)^2-\partial_t(\eta w)^2)
-2\eta w\grad \eta\grad w-\eta w^2\Delta \eta
\end{equation}
and
\begin{eqnarray}\label{inter6.2a}
\varphi'(r) &=& -\frac{1}{2r^2}\int_{-r^2}^0\int_{\R^n_+}(\Delta (\eta w)^2-\partial_t(\eta w)^2)G(x,-s)dxds+2\int_{\R^n_+}|\grad(\eta w)(x,-r^2)|^2G(x,r^2)dx
\nonumber\\ 
&+&\frac{1}{r^2}\int_{-r^2}^0\int_{\R^n_+}(
2\eta w \grad \eta\grad w+\eta w^2 \Delta\eta)dxdt.
\end{eqnarray}

We integrate by parts to obtain
\begin{eqnarray}\label{inter6.3a}
\varphi'(r) &=& \frac{1}{2r^2}\int_{-r^2}^0\int_{\R^n_+}(\grad(\eta w)^2\grad G+\partial_t(\eta w)^2G)dxds-\frac{1}{2r^2}\int_{-r^2}^0\int_{\R^{n-1}}[(\eta w)^2]_\nu(x',0,s)G(x',0,-s)dx'ds\nonumber\\ 
&+&\frac{1}{r^2}\int_{-r^2}^0\int_{\R^n_+}(
2\eta\grad w\grad \eta+\eta w^2\Delta \eta)G(x,-s)dxds+2\int_{\R^n_+}|\grad(\eta w)(x,-r^2)|^2G(x,r^2)dx.\nonumber\\
\end{eqnarray}

Integrating again by parts, we obtain
\begin{eqnarray}\label{inter6.4a}
\varphi'(r) &=& -\frac{1}{2r^2}\int_{-r^2}^0\int_{\R^n_+}(\eta w)^2(\Delta G+\partial_tG)dxds-\frac{1}{2r^2}\int_{-r^2}^0\int_{\R^{n-1}}(\eta w)^2_\nu G(x,-s)dx'ds\nonumber\\ 
&-&\frac{1}{2r^2}\int_{\R^n_+}(\eta w)^2(x,-r^2)G(x,r^2)dx+\frac{1}{r^2}\int_{-r^2}^0\int_{\R^n_+}(
\eta \Delta \eta w^2
+2\eta w\grad\eta\grad w)G(x,-s)dxds\nonumber\\ 
&+&2\int_{\R^n_+}|\grad(\eta w)(x,-r^2)|^2G(x,r^2)dx.
\end{eqnarray}
Since $w(0,0)=0$, we have
\begin{eqnarray}\label{inter6.5a}
\varphi'(r) &=&-\frac{1}{2r^2}\int_{\R^n_+}(\eta w)^2(x,-r^2)G(x,r^2)dx+2\int_{\R^n_+}|\grad(\eta w)(x,-r^2)|^2G(x,r^2)dx \nonumber\\
&-&\frac{1}{2r^2}\int_{-r^2}^0\int_{\R^{n-1}}2\eta w\eta w_\nu G(x,-s)dx'ds
+\frac{1}{r^2}\int_{-r^2}^0\int_{\R^n_+}\eta
\Delta \eta w^2G(x,-s)dxds\nonumber\\
&+&\frac{2}{r^2}\int_{-r^2}^0\int_{\R^n_+}\eta w \grad\eta\grad w G(x,-s)dxds \nonumber\\
\end{eqnarray}
or
\begin{eqnarray}\label{inter6.6a}
\varphi'(r) &=&-\frac{1}{2r^2}\int_{\R^n_+}(\eta w)^2(x,-r^2)G(x,r^2)dx+2\int_{\R^n_+}|\grad(\eta w)(x,-r^2)|^2G(x,r^2)dx \nonumber\\
&+&\frac{1}{2r^2}\int_{-r^2}^0\int_{\R^n_+}\grad \eta^2\grad w^2 G(x,-s)dxds+\frac{1}{r^2}\int_{-r^2}^0\int_{\R^n_+}\eta \Delta \eta w^2 G(x,-s)dxds\nonumber\\
&+&\frac{1}{r^2}\int_{-r^2}^0\int_{\R^{n-1}}\eta^2wf \ G(x,-s)dx'ds
\end{eqnarray}
and finally
$$\varphi'(r)\geq -\frac{1}{2r^2}\int_{\R^n_+}(\eta w)^2(x,-r^2)G(x,r^2)dx+2\int_{\R^n_+}|\grad(\eta w)(x,-r^2)|^2G(x,r^2)dx-Cr^\alpha.$$

Now, consider the truncated function $\overline{w}=-(w+r^{\alpha+\delta})^-$ and note that
$$\int_{\R^n_+}|\grad(\eta \overline{w})(x,-r^2)|^2G(x,r^2)dx\leq \int_{\R^n_+}|\grad(\eta w)(x,-r^2)|^2G(x,r^2)dx.$$
Hence
$$\varphi'(r)\geq -\frac{1}{2r^2}\int_{\R^n_+}[\eta (w-\overline{w})+\eta\overline{w}]^2(x,-r^2)G(x,r^2)dx+2\int_{\R^n_+}|\grad(\eta \overline{w})(x,-r^2)|^2G(x,r^2)dx-Cr^\alpha$$
and
$$\varphi'(r)\geq -\frac{1}{2r^2}\int_{\R^n_+}\eta^2[(w-\overline{w})^2+2\overline{w}(w-\overline{w})]G(x,r^2)dx-Cr^\alpha$$
or
$$\varphi'(r)\geq-\frac{3}{2}r^{2\alpha+2\delta-2}-Cr^\alpha\geq -\frac{3}{2}r^{2\alpha+\delta-2}.$$
Therefore
$$\varphi(1)-\varphi(r)\geq -\frac{3}{2}\bigg(\frac{1}{2\alpha+\delta-1}\bigg)+\frac{3}{2}\bigg(\frac{1}{2\alpha+\delta-1}\bigg)r^{2\alpha+\delta-1}.$$
Since $\varphi(1)$ is universally bounded the proof is complete.

\end{proof}

Next, we state our main result of this subsection:

\begin{thm}\label{Theorem5}
Let $u$ the solution of (\ref{model2}), then $\grad u$ is $C_{x,t}^{\frac{1}{2},\frac{1}{4}}$ up to the hyperplane $\R^{n-1}$.
\end{thm}
\begin{proof}
Let $w=u_{x_n}$ and $\overline{w}$ be as in the proof of Lemma \ref{Lemma6non-a}. Fix $s>0$, choose $R>0$ large enough and $\varepsilon<s$. We define a cut-off function $\eta=\eta(x)$ so that $\text{supp}\eta\in B_{R+1}(0)$, $\eta \equiv 1$ on $B_R(0)$ and $|\grad\eta|\leq C$. 

Then 
\begin{equation}\label{thm5-1}
(\Delta-\partial_\xi)(\eta^2\overline{w})=2\eta^2|\grad \overline{w}|^2+4\overline{w}\eta\grad\overline{w}\grad\eta
+2(\eta\Delta\eta+|\grad\eta|^2)\overline{w}^2+2\eta^2\overline{w}(\Delta\overline{w}-\partial_\xi\overline{w}).
\end{equation}
Recall that $(\Delta+\partial_\xi)G(x,-\xi)=\delta_{(0,0)}$, therefore using (\ref{thm5-1}), an integration by parts along with the fact that $\eta$ is compactly supported we obtain
$$2\int_{-s}^{-\varepsilon}\int_{\R^n_+}\eta^2|\grad\overline{w}|^2G(x,-\xi)dxd\xi=-\int_{\R^n_+}\eta^2\overline{w}^2G(x,\varepsilon)dx+\int_{\R^n_+}\eta^2\overline{w}^2G(x,s)dx$$ 
$$-4\int_{-s}^{-\varepsilon}\int_{\R^n_+}\overline{w}
\eta\grad\eta\grad
\overline{w}G(x,-\xi)dxd\xi-2\int_{-s}^{-\varepsilon}\int_{\R^n_+}(\eta\Delta\eta+|\grad\eta|^2)\overline{w}^2G(x,-\xi)dxd\xi$$
\begin{equation}\label{thm5-2}
-2\int_{-s}^{-\varepsilon}\int_{\R^n_+}\eta^2\overline{w}(\Delta \overline{w}-\partial_{\xi}\overline{w})G(x,-\xi)dxd\xi.
\end{equation}
Observe that $$\int_{-s}^{-\varepsilon}\int_{\R^n_+}\overline{w}\eta|\grad\eta||\grad\overline{w}|G(x,-\xi)dxd\xi\leq C\int_{-s}^{-\varepsilon}\int_{B_{R+1}^+\setminus B_{R}^+}|\overline{w}||\grad\overline{w}|\frac{e^{-R^2/4|\xi|}}{|\xi|^{n/2}}dxd\xi$$
$$\leq Ce^{-R^2/4+\varepsilon_0}\int_{-s}^{0}\int_{B_{R+1}^+\setminus B_{R}^+}|\overline{w}||\grad\overline{w}|dxd\xi.$$
Using Cauchy-Schwartz, we conclude that the last three terms on the right hand side of (\ref{thm5-2}) behave the same, in particular they decay to zero as $R\rightarrow \infty$. Therefore we conclude that
$$(\eta\overline{w})^2(0,0)\leq \int_{\R_+^n}(\eta\overline{w})^2G(x,s)dx$$
or, after rescaling,
\begin{equation}\label{thm5-3}
(\eta\overline{w})^2(x,t)\leq \int_{\R^n_+}(\eta\overline{w})^2(y,s)G(x-y,t-s)dy.
\end{equation}
for every $(x,t)\in Q^+_{r/2}$ and $-r^2<s<-\frac{r^2}{2}$. By Poincar\'{e} inequality for Gaussian measures (see \cite{Beckner}) we have that
\begin{equation}\label{thm5-4}
\int_{\R^n_+}(\eta\overline{w})^2(y,s)G(x-y,t-s)dy\leq 2 |s|\int_{\R^n_+}|\grad(\eta\overline{w})(y,s)|^2G(x-y,t-s)dy
\end{equation}
for $(x,t)\in Q^+_{r/2}$ and $-r^2<s<-\frac{r^2}{2}$. Combine (\ref{thm5-3}) and (\ref{thm5-4})
to obtain
\begin{equation}\label{thm5-5}
(\eta\overline{w})^2(x,t)\leq C |s|\int_{\R^n_+}|\grad(\eta\overline{w})(y,s)|^2G(x-y,t-s)dy
\end{equation}
for every $(x,t)\in Q^+_{r/2}$ and $-r^2<s<-\frac{r^2}{2}$. An integration with respect to $s$ in (\ref{thm5-5}) shows that
$$(\eta\overline{w})^2(x,t)\leq C\int_{-r^2}^{-r^2/2}\int_{\R^n_+}|\grad(\eta\overline{w})(y,s)|^2G(x-y,t-s)dyds$$
for every $(x,t)\in Q^+_{r/2}$. Now the dichotomy for $\varphi(r)$ in Lemma \ref{Lemma6non-a} provides a $C^{1/2}$ modulus of continuity for $w$, as in the proof of Theorem 5 in \cite{AC}.
\end{proof}

\subsection{H\"{o}lder continuity of the time derivative near a free boundary point of positive parabolic density}\label{optimal time}

Although the positive time derivative is always H\"{o}lder continuous
(see \S\ref{subsection4.2}), one does not expect to obtain continuity of the full time derivative without further restrictions. The purpose of this section is to show that, indeed, H\"{o}lder continuity of the full time derivative can be achieved near free boundary points of positive parabolic density with respect to the coincidence set. In order to achieve this desired result we employ the  well known "hole filling" method of Widman (see \cite{Widman}) adapted for parabolics by Struwe (see \cite{Struwe}). As it was mentioned in the introduction,  the results of the present section are independent of the quasi-convexity.  

\begin{defn}\label{positive density}
A free boundary point $(x_0',0,t_0)$ is of positive parabolic density with respect to the coincidence set if there exist positive constants $c>0$ and $r_0>0$ such that $|Q_r'(x'_0,0,t_0)\cap\{u=0\}|\geq c|Q'_r(x'_0,0,t_0)|$ $ \forall r<r_0$.
\end{defn}

So the main result of this subsection is stated as follows::

\begin{thm}\label{Theorem6}
Let $(x_0,t_0)$ be a free boundary point of positive parabolic density with respect to the coincidence set to problem (\ref{model2}). Then $u_t$ is H\"{o}lder continuous in a neighborhood of $(x_0,t_0)$.
\end{thm} 

\begin{proof}Since, by \S\ref{subsection4.2}, $u_t^+$ is H\"{o}lder continuous, it suffices to prove the theorem for $u_t^-$. Actually, we will show that $u_t^-$ decays to zero in parabolic cylinders shrinking to the free boundary point $(x_0,t_0)$. We consider the penalized solution $u^\varepsilon$ of (\ref{model2}) in $Q_r^+(x_0,t_0)$ with $r<r_0$, where $r_0$ is as in Definition \ref{positive density}. For simplicity we take $(x_0,t_0)=(0,0)$ and $r=1$. Differentiate with respect to $t$ to have as in (\ref{4.2.2}) 
\begin{equation}\label{penalized equation}
\begin{cases}  
\Delta v^\varepsilon-\partial_t v^\varepsilon=f^\varepsilon_t,   &{\rm{in}} \ \  Q_1^+ \cr  -\partial_\nu v^\varepsilon
=\beta'_\varepsilon(u^\varepsilon)v^\varepsilon  &{\rm{on}} \ \  Q_1' \cr
\end{cases}
\end{equation}
where  $v^\varepsilon:=(u^\varepsilon)_t$. 
For any $(\xi,\tau)\in Q^+_\frac{1}{5}$ we want to multiply the equation by an appropriate test function and integrate by parts over the set 
$Q_\frac{3}{5}^+(\xi,\tau):=Q_\frac{3}{5}(\xi,\tau)\cap \{x_n\geq0\}\subset Q^+_1$. 
This will lead us to an estimate which will iterated to yield the desired result.

The aforesaid appropriate test function will be the product of following three functions: 

The first one is the square of a smooth function $\zeta(x,t)$ supported in 
$Q_\frac{3}{5}^+(\xi,\tau)$ such that 
$\zeta\equiv 1$ for every $(x,t)\in Q_\frac{2}{5}^+(\xi,\tau)$, $|\grad\zeta|\leq c$\ with\ 
$supp(\grad\zeta)\subset (B_\frac{3}{5}^+(\xi,\tau)\setminus B_\frac{2}{5}^+(\xi,\tau))\times(\tau-\frac{9}{25},\tau],
\ 0\leq \zeta_t \leq c$  with
$\ supp (\zeta_t)\subset B_\frac{3}{5}(\xi,\tau) \times (\tau-\frac{9}{25},\tau-\frac{4}{25})$.

 The second one is a smoothing of the fundamental solution $G(x,t)$ of the heat equation (see (\ref{fundamental solution})), i.e.
$$G_{\delta}^{(\xi,\tau)} (x,t):=
(G(x-\xi,\tau-t)\chi(x,t)_{E^c_{\delta}(\xi,\tau)}+p(x-\xi,t-\tau)\chi(x,t)_{E_{\delta}(\xi,\tau)})\chi(x,t)_{\{t<\tau\}}$$ 
where $\ E_\delta(\xi,\tau):=\{(x,t)\in R^{n+1}:t\leq\tau,\ 
  G(x-\xi,t-\tau)\geq\frac{1}{\delta^n}\}$,  the "heat" ball of "radius" $\delta$ about $(\xi,\tau)$,
and $p(x,t):=\frac{1}{\delta^n}(\frac{|x|^2}{4t}+\log\frac{e\delta^n}{(4\pi|t|)^{n/2}})\chi(x,t)_{\{t<0 \}}$.
Notice that $G^{(\xi,\tau)}_{\delta}$ is a $C^1$ function everywhere in $\R^{n+1}$ except at 
$(\xi,\tau)$.  In order to deal with this problem we just translate the singularity outside of our domain by a small amount $\varepsilon'>0$ and then we let $\varepsilon'$ to tend to zero, for simplicity we omit this technicality.

Finally the third function is $(v^{\varepsilon})^-$ which can be smoothed out by the standard way; again we omit it for the sake of simplicity.

Therefore we multiply the equation in (\ref{penalized equation}) by $\zeta^2 G_{\delta}^{(\xi,\tau)} (v^{\varepsilon})^-$  and integrate by parts over 
$Q_\frac{3}{5}^+(\xi,\tau)$ to obtain

$$\int_{Q_\frac{3}{5}^+(\xi,\tau)}(\grad (\zeta^2 G_{\delta}^{(\xi,\tau)} (v^{\varepsilon})^-) \grad v^\varepsilon+(\zeta^2 G_{\delta}^{(\xi,\tau)} (v^{\varepsilon})^-)\partial_tv^\varepsilon)dxdt
=-\int_{Q'_\frac{3}{5}(\xi,\tau)}(\zeta^2 G_{\delta}^{(\xi,\tau)} (v^{\varepsilon})^-)\beta'(u^{\varepsilon})
v^{\varepsilon} dx'dt$$
\begin{equation}
-\int_{Q_\frac{3}{5}^+(\xi,\tau)}\zeta^2 G_{\delta}^{(\xi,\tau)} (v^{\varepsilon})^- f_tdxdt
\end{equation}
By calculating appropriately and by noticing that due to the non negativity of $\beta'_\varepsilon$ the boundary integral term has the right sign, so it can be omitted, we obtain
$$\int_{Q_\frac{3}{5}^+(\xi,\tau)}(G_{\delta}^{(\xi,\tau)}|\grad (\zeta (v^\varepsilon)^-)|^2
+\frac{1}{2}[\grad G_{\delta}^{(\xi,\tau)}\grad (\zeta (v^\varepsilon)^-)^2
+G_{\delta}^{(\xi,\tau)}\partial_t (\zeta (v^\varepsilon)^-)^2])dxdt$$
$$\leq \int_{Q_\frac{3}{5}^+(\xi,\tau)} G_{\delta}^{(\xi,\tau)}(|\grad \zeta|^2+\zeta \zeta_t)((v^\varepsilon)^-)^2dxdt
+\frac{1}{2}\int_{Q_\frac{3}{5}^+(\xi,\tau)}\grad G_{\delta}^{(\xi,\tau)}\grad \zeta^2((v^\varepsilon)^-)^2dxdt$$ 
$$+\int_{Q_\frac{3}{5}^+(\xi,\tau)} \zeta^2 G_{\delta}^{(\xi,\tau)}(v^\varepsilon)^- f_tdxdt.$$  
Using the fact that 
$\supp (\mu)=E_\delta(\xi,\tau)$ 
where 
$\mu:=-(\Delta + \partial_t)G_{\delta}^{(\xi,\tau)}$ 
with 
$d\mu=\frac{1}{4\delta^n}dE_{\delta}(\xi,\tau)$ 
and 
$|E_{\delta}(\xi,\tau)|=4\delta^n$ 
(see \cite{Evans}) and that, for $\delta$ small enough, the inequalities
$0\leq G_{\delta}^{(\xi,\tau)}\leq C(n)\ \ in\ \   (B_{\frac{4}{5}}^+\setminus B_\frac{1}{5}^+)\times(-\frac{2}{5},0)$,
and  
$c(n)\leq G_{\delta}^{(\xi,\tau)} \leq C(n)\ \ in\ \ B_{\frac{4}{5}}^+\times (-\frac{2}{5},-\frac{4}{25}),$ 
we have
$$\int_{Q_\frac{2}{5}^+(\xi,\tau)} G_{\delta}^{(\xi,\tau)}|\grad (v^{\varepsilon})^-|^2dxdt
+ \fint_{E_{\delta}^+(\xi,\tau)} (v^{\epsilon})^-dE_{\delta}(\xi,\tau)
\leq C(n)  \int_{-\frac{2}{5}}^{-\frac{4}{25}} \int_{B_\frac{4}{5}^+}((v^\varepsilon)^-)^2dxdt$$
\begin{equation}\label{estimate 1}
+C(n)\int_{-\frac{2}{5}}^0 \int_{B_\frac{4}{5}^+\setminus B_\frac{1}{5}^+}((v^\varepsilon)^-)^2dxdt+ C(n)M
\end{equation}
where 
$M:=||v^\varepsilon||_\infty ||f_t||_\infty$.
 
Now, we first let $\varepsilon$ tend to $0$ in order to obtain (\ref{estimate 1}) for $v^-$, then we let $\delta$ to go to $0$, and, finally, we take the supremum over 
$(\xi,\tau)\in Q_\frac{1}{4}^+$ to obtain, a fortiori,
\begin{equation}\label{inequality for iteration}
\int_{Q_\frac{1}{5}^+}G(x,-t)|\grad v^-|^2dxdt
+\sup_{Q_\frac{1}{5}^+}(v^-)^2
\leq C(n)\int_{-\frac{2}{5}}^{-\frac{4}{25}}\int_{B_\frac{4}{5}^+}(v^-)^2dxdt
+C\sup_{Q_1^+\setminus Q_\frac{1}{5}^+}(v^-)^2+CM.
\end{equation}
Next we want to control the first integral of the right hand side of (\ref{inequality for iteration}) by one similar to the first integral of the left hand side of (\ref{inequality for iteration}). To do that we first multiply the equation in (\ref{penalized equation}) by $\zeta^2(v^\varepsilon)^-$ where $\zeta$ is a smooth cutoff function supported in $B_1\times(-1,t)$ , for any $t\leq-\frac{4}{25}$, $\zeta \equiv 1$ on $B_\frac{4}{5} \times (-\frac{2}{5},t)$, and vanishing near its parabolic boundary with $|\grad\zeta|\leq c$ and $0\leq\zeta_t\leq c$, then we integrate by parts over this set intersected by $\R^n_+$ to have
$$\int_{-1}^t\int_{B_1^+}(\grad (\zeta^2(v^{\varepsilon})^-) \grad v^\varepsilon+(\zeta^2
(v^{\varepsilon})^-)\partial_tv^\varepsilon)dxdt
=-\int_{-1}^t\int_{B_1'}(\zeta^2 (v^{\varepsilon})^-)\beta'(u^{\varepsilon})
v^{\varepsilon} dx'dt$$ 
$$-\int_{-1}^t\int_{B_1^+}\zeta^2
(v^{\varepsilon})^- f_tdxdt.$$
Again, exploiting the positivity of $\beta'$ and letting $\varepsilon$ go to zero, we arrive, as above but in a much simpler way, at the following inequality
$$\int_{B_ \frac{4}{5}^+}(v^-)^2(x,t)dx
+\int_{-\frac{2}{5}}^t\int_{B_\frac{4}{5}^+}|\grad  v^-|^2dxdt
\leq c
\int_{Q_1^+}(v^-)^2dxdt+C(n)Mr^{n+2}$$
$\forall\ t \in (-\frac{2}{5},-\frac{4}{25})$. Observe that
a sufficient portion of the coincidence set is present in $Q_1$ so that the parabolic Poincar\'{e} inequality can be applied to dominate the integral on the right hand side of the above inequality. Therefore, since the second term on the left hand side is non negative, we have, for every 
$-\frac{2}{5}\leq t\leq-\frac{4}{25}$,
$$\int_{B_ \frac{4}{5}^+}(v^-)^2(x,t)dx
\leq C(n)\int_{Q_1^+}|\grad v^-|^2dxdt
+C(n)M.$$
We then integrate the above inequality with respect to $t$ from $-\frac{2}{5}$ to $-\frac{4}{25}$ to get
$$\int_{-\frac{2}{5}}^{-\frac{4}{25}}
\int_{B_ \frac{4}{5}^+}(v^-)^2dxdt
\leq C(n)\int_{Q_1^+}|\grad(v^-)|^2dxdt
+C(n)M.$$
Insert this in (\ref{inequality for iteration}) above and, using the fact that $G(x,-t)\geq c(n)$ for 
$-\frac{2}{5}\leq t\leq-\frac{4}{25}$, 
to have  
$$\int_{Q_\frac{1}{5}^+}G(x,-t)|\grad v^-|^2dxdt
+\sup_{Q_\frac{1}{5}^+}(v^-)^2\leq$$
\begin{equation}\label{hole iteration}
C(n)(\int_{Q_1^+ \setminus Q_\frac{1}{5}^+}
G(x,-t)|\grad v^-|^2dxdt
+\sup_{Q_1^+\setminus Q_\frac{1}{5}^+}(v^-)^2))
+C'(n)M.
\end{equation}
Set 
$\omega(\rho):=\int_{Q_\rho^+}G|\grad v^-|^2dxdt
+\sup_{Q_\rho^+}(v^-)^2$, 
then add 
$C(n)\omega(\frac{1}{5})$ 
to both sides of (\ref{hole iteration}) and divide the new inequality by $1+C(n)$ to have
\begin{equation}\label{iteration 1}
\omega(\frac{1}{5})\leq \lambda \omega (1)+c
\end{equation}
where $\lambda =\frac{C(n)}{1+C(n)}$. Iteration of (\ref{iteration 1}) implies that there exists an $\alpha=\alpha(\lambda)\in (0,1)$ and a constant 
$C=C(n,||u_t||_{\infty},||f_t||_\infty)$ such that
$$\omega (\rho)\leq C\rho^\alpha$$  
for every $0<\rho\leq\frac{r_0}{5}$. This concludes the H\"{o}lder continuity from the past. The continuity from the future follows, now, by standard methods. 
\end{proof}

\subsection{Free boundary regularity}\label{free boundary}

In the study of free boundary regularity it turns out that in order to achieve smoothness of the free boundary one has to focus his attention in a neighborhood of certain free boundary points, which we shall call them non-degenerate, (see Definition \ref{nondeg-point} below). A good candidate for a non-degenerate free boundary point must include one of positive parabolic density of the coincidence set. The fact, that $u_t$ is H\"{o}lder continuous at such a point (see \S\ref{optimal time}), yields a control of the speed of the interphase, a crucial step for our further analysis of the regularity of the free boundary. 
 Since it is more convenient to work with the zero obstacle and with the right hand side of the equation to vanish at the point, which, for simplicity, we take it to be the origin,
we set  $\tilde{u}(x',x_n,t)=u(x',x_n,t)-\psi(x',t) +\frac{1}{2}H\psi(0,0)x^2_n
$ ($H:=\Delta-\partial_t$). Observe that $\{\tilde{u}(x',x_n,t)=0\}=\{u(x',x_n,t)=\psi(x',t)\}$ and upon reflection $\tilde{u}$ in $B^*_1:=\{(x,t)\in \R^{n+1}:|x|^2+t^2<1\}$ satisfies:

\begin{equation}\label{4.28}
\begin{cases}
\tilde{u}(x',0,t)\geq 0 & {\rm{in}}\ \ B^*_1\cap\{x_n=0\}\\
\tilde{u}(x',x_n,t)=\tilde{u}(x',-x_n,t) & {\rm{in}} \ \ B^*_1\\
\Delta\tilde{u}(x',x_n,t)-\partial_t\tilde{u}(x',x_n,t)=H\psi(0,0)-H\psi(x',t) & {\rm{in}}\ \ B^*_1\setminus\{\tilde{u}=0\}\\
\Delta\tilde{u}(x',x_n,t)-\partial_t\tilde{u}(x',x_n,t) \leq H\psi(0,0)-H\psi(x',t) & {\rm{in}}\ \ B^*_1 \\
\end{cases}
\end{equation}

For simplicity of notation we\; "drop"\; the\; "$\sim$"\; for the rest of this section.

Now we pass the $u_t$ term to the right hand side of the equation and if we assume that $H\psi$ is at least $C^{\alpha}$ we can apply the elliptic theory developed in  \cite{ACS2008}, \cite{CSS} and extended in \cite{CDS} at the t-level of the point. Consequently, if the origin is regular then at $t=0$  
the blow up limit $v_0$ of the solution $u$ (up to sub-sequences) exists, 
and, in appropriate coordinates,
$$v_0(x)=\frac{2}{3}\rho^\frac{3}{2}\cos(\frac{3}{2}\theta)$$ 
where $\rho=\sqrt{x_1^2+x_n^2}$ and $\theta={\arctan}(\frac{x_n}{x_1})$ (unique up to rotations).

Now, we are ready to state the "hyperbolic" definition of our non-degenerate free boundary point.

\begin{defn}\label{nondeg-point}
Let $(x_0,t_0)$ be a free boundary point and 
$B_r^*(x_0,t_0):=\{(x,t)\in R^{n+1}: (x-x_0)^2+(t-t_0)^2<r^2\}$, set $$l:=\limsup_{r\to 0^+}\frac{||u||_{L^\infty(B_r^*(x_0,t_0))}}{r^{3/2}}$$
A point $(x_0,t_0)$ is called a non-degenerate free boundary point if it is of positive parabolic density of the coincidence set and $0<l<\infty$, otherwise degenerate.
\end{defn}

With this definition at our hands we state the main result of this section:

\begin{thm}
Let $u$ be a solution to (\ref{4.28}). Assume the origin to be a non-degenerate free boundary point. Then the free boundary is a $C^{1,\alpha}$ $n$-dimensional surface about the origin.
\end{thm}

The following "hyperbolic" blow up sequence will be very useful for our analysis since, at a point, it preserves the geometry of the free boundary: $$u_r(x,t):=\frac{u(rx,rt)}{r^{3/2}}.$$

\begin{lemma}\label{blowups}
Let $u$ be a solution to (\ref{4.28}). If $(0,0)$ is a non-degenerate free boundary point then there exists a sequence  $u_{r_j}$ of blow ups which converges uniformly on compact subsets to a function $u_0$ such that, (in appropriate coordinates),  
$$u_0(x,t)=\frac{2}{3}\rho(t)^\frac{3}{2}\cos(\frac{3}{2}\theta(t))$$ 
where 
$\rho(t):=\sqrt{(x_1+\omega t)^2+x_n^2}$ 
and 
$\theta(t):=\arctan(\frac{x_n}{x_1+\omega t})$ for some $\omega\in\R$. 
\end{lemma}
\begin{proof}
Since $0<l<\infty$, it is clear that we can extract a subsequence $u_{r_j}$ converging uniformly on compact subsets to a non trivial limit $u_0$. This $u_0$ is a harmonic function for every fixed $t$ outside of the coincidence set; and the coincidence set, due to the density assumption, is a convex cone in $\R^n$, or more precisely in 
$(x',t)$ 
variables. Also, by the discussion above, at $t=0$ 
$u_0=\frac{2}{3}\rho^{3/2}\cos\frac{3}{2}\theta$ 
where $\rho=\sqrt{x_1^2+x_n^2}$ and 
$\theta=\arctan(\frac{x_n}{x_1})$. Moreover the convex cone is composed by the following two supporting hyperplanes $Ax_1+at=0$ for $t\geq0$ and $Bx_1+bt=0$ for $t\leq0$ with the constants $A\geq0$, $B\geq0$ and $bA\leq{aB}$. 
We want to prove that this convex cone is actually a non-horizontal half space i.e. $A>0$, $B>0$, and $bA=aB$, and $u_0$ admits the stated representation; we do this in several steps: 

\textit{Step I: $A>0$ and $B>0$}

For, if $A=0$ then for every $t>0$ $u_0(x,t)$ is harmonic in all of $\R^n$ i.e. of polynomial growth. But for $t=0$ $u_0$ has $3/2$ degree of growth, therefore, by continuity of $u_0$, a contradiction. Similarly $B>0$.

\textit{Step II: For each fixed t, $u \sim |x|^\frac{3}{2}$ as $|x| \to \infty$ with $x\cdot e_1\geq \varepsilon$ for some $\varepsilon>0$}

It is enough to show the bound by below. Therefore take a sequence  $x^{(j)}$ such that $|x^{(j)}| \to \infty$ with $x^{(j)} \cdot e_1 \geq \varepsilon$ for every $j$ then by convexity $u_0(x^{(j)},t) \geq u_0(x^{(j)},0) + (u_0)_t(x^{(j)},0)t$, hence by the  behavior of $u_0$ at $t=0$ the result follows.

\textit{Step III: For each fixed $t$, 
$$u_0(x,t) = \frac{2}{3} \rho(t)^\frac{3}{2} \cos (\frac{3}{2}\theta (t)),$$
where for 
$t>0$, 
$\rho(t)=\sqrt{(x_1+\frac{a}{A}t)^2+x_n^2}$, 
$\theta(t) = \arctan\frac{x_n}{x_1+\frac{a}{A}t}$
and for 
$t<0$, 
$\rho(t)=\sqrt{(x_1+\frac{b}{B}t)^2+x_n^2}$,
$\theta (t) = \arctan \frac{x_n}{x_1+\frac{b}{B}t}$
}

Indeed, for each fixed 
$t>0$, $u_0$ 
is a harmonic function which vanishes for 
$\{x_1\leq -\frac{a}{A}t\}\cap\{x_n=0\}$
and grows at infinity with 
$\frac{3}{2}$ 
exponent, therefore by Phragmen-Lindelof theorem we obtain the representation. Analogously, for $t<0$.  

\textit{Step IV: $bA=aB$}

For, if not then 
$$\partial_t u_0(0,0^+)-\partial_t u_0(0,0^-)=(\frac{a}{A}-\frac{b}{B})\rho^\frac{1}{2}\cos(\frac{1}{2}\theta)\neq0,$$ 
whence, by approximation, a contradiction to the continuity of $\partial_t u$ at the origin.

Set $\omega:=\frac{a}{A}$ and the proof is complete.
\end{proof}

Finally we prove our theorem:
\begin{proof}
Obviously the existence of $\omega$ in Lemma \ref{blowups} implies the differentiability of the free boundary at the origin. Also, due to the upper semi-continuity of the elliptic Almgren's frequency function, we have the differentiability of the free boundary for any nearby point $p=(x_p,t_p)$ at least when $t_p\leq 0$, since $u_t$ is continuous there. Now, if $t_p>0$ and $p=(x_p,t_p)$ still near the origin, we observe that the frequency function will converge to $\frac{3}{2}$. Consequently, the point $p=(x_p,t_p)$ will be a free boundary point of positive parabolic density with respect to zero set, which renders $u_t$ continuous there. Hence we have the differentiability of the free boundary there, too. To prove the continuous differentiability of it consider two distinct free boundary points nearby, say p and 0. Assume, on the contrary, that it is not true, that is $\omega(p)$ does not converge to $\omega(0)$ as $p \to 0$. Consider the blow up sequences  $u_{r_j}^{(p)}$ and $u_{r_i}^{(0)}$ around $p$ and $0$, respectively, where $u_{r_j}^{(p)}(x,t):=\frac{u(r_j((x,t)-p))}{r_J^3/2}$. 
These sequences converge uniformly to 
$$u_0^{(p)}(x,t):= \frac{2}{3}\rho^\frac{3}{2}(p,t)\cos \frac{3}{2}\theta (p,t)$$
and 
$$u_0^{(0)}(x,t):= \frac{2}{3}\rho^\frac{3}{2}(0,t)\cos \frac{3}{2}\theta (0,t)$$
respectively, where 
$\rho(p,t):=\sqrt{(x_1(p)+\omega(p)t(p))^2+x_n^2)}$ and $\theta(p,t):=\arctan\frac{x_n}{x(p))+\omega(p)t(p)}$. So, if $\omega(p)$ does not converge to $\omega(0)$ then $u_0^{(p)}$ does not converge to $u_0^{(0)}$, therefore a contradiction to the continuity of the solution $u$. Hence a $C^{\alpha}$ estimate of the free boundary normals follows easily. 
\end{proof}

\bibliographystyle{plain}   
\bibliography{biblio}             

\def\cprime{$'$}
\begin{thebibliography}{10}

\bibitem{ArU}
A.~Arkhipova and N.~Ural{\cprime}tseva.
\newblock Regularity of the solution of a problem with a two-sided limit on a
  boundary for elliptic and parabolic equations.
\newblock {\em Trudy Mat. Inst. Steklov.}, 179:5--22, 241, 1988.
\newblock Translated in Proc. Steklov Inst. Math. {{\bf{1}}989}, no. 2, 1--19,
  Boundary value problems of mathematical physics, 13 (Russian).

\bibitem{Athparabolic}
I.~Athanasopoulos.
\newblock Regularity of the solution of an evolution problem with inequalities
  on the boundary.
\newblock {\em Comm. Partial Differential Equations}, 7(12):1453--1465, 1982.

\bibitem{A}
I.~Athanasopoulos.
\newblock A temperature control problem.
\newblock {\em Internat. J. Math. Math. Sci.}, 7(1):113--116, 1984.

\bibitem{AC}
I.~Athanasopoulos and L.~Caffarelli.
\newblock Optimal regularity of lower dimensional obstacle problems.
\newblock {\em Zap. Nauchn. Sem. S.-Peterburg. Otdel. Mat. Inst. Steklov.
  (POMI)}, 310(Kraev. Zadachi Mat. Fiz. i Smezh. Vopr. Teor. Funkts. 35
  [34]):49--66, 226, 2004.

\bibitem{AC2010}
I.~Athanasopoulos and L.~Caffarelli.
\newblock Continuity of the temperature in boundary heat control problems.
\newblock {\em Advances in Math.}, 224(1):293--315, 2010.

\bibitem{ACMPart2}
I~Athanasopoulos, L.~Caffarelli, and E.~Milakis.
\newblock {O}bstacle problems for parabolic nonlocal operators.
\newblock {\em Preprint}, 2015.

\bibitem{ACS2008}
I.~Athanasopoulos, L.~Caffarelli, and S.~Salsa.
\newblock The structure of the free boundary for lower dimensional obstacle
  problems.
\newblock {\em Amer. J. Math.}, 130(2):485--498, 2008.

\bibitem{Beckner}
W.~Beckner.
\newblock A generalized {P}oincar\'e inequality for {G}aussian measures.
\newblock {\em Proc. Amer. Math. Soc.}, 105(2):397--400, 1989.

\bibitem{CDS}
L.~Caffarelli, D.~De~Silva, and O.~Savin.
\newblock The two membranes problem for different operators.
\newblock {\em Preprint.}, 2015.

\bibitem{CaffEvans}
L.~Caffarelli and L.~C. Evans.
\newblock Continuity of the temperature in the two-phase {S}tefan problem.
\newblock {\em Arch. Rational Mech. Anal.}, 81(3):199--220, 1983.

\bibitem{CFi}
L.~Caffarelli and A.~Figalli.
\newblock Regularity of solutions to the parabolic fractional obstacle problem.
\newblock {\em J. Reine Angew. Math.}, 680:191--233, 2013.

\bibitem{CFr}
L.~Caffarelli and A.~Friedman.
\newblock Continuity of the temperature in the {S}tefan problem.
\newblock {\em Indiana Univ. Math. J.}, 28(1):53--70, 1979.

\bibitem{CPS}
L.~Caffarelli, A.~Petrosyan, and H.~Shahgholian.
\newblock Regularity of a free boundary in parabolic potential theory.
\newblock {\em J. Amer. Math. Soc.}, 17(4):827--869, 2004.

\bibitem{CSS}
L.~Caffarelli, S.~Salsa, and L.~Silvestre.
\newblock Regularity estimates for the solution and the free boundary of the
  obstacle problem for the fractional {L}aplacian.
\newblock {\em Invent. Math.}, 171(2):425--461, 2008.

\bibitem{CS}
L.~Caffarelli and L.~Silvestre.
\newblock An extension problem related to the fractional {L}aplacian.
\newblock {\em Comm. Partial Differential Equations}, 32(7-9):1245--1260, 2007.

\bibitem{CaffVass}
L.~Caffarelli and A.~Vasseur.
\newblock Drift diffusion equations with fractional diffusion and the
  quasi-geostrophic equation.
\newblock {\em Ann. of Math.}, 171(2):1903--–1930, 2010.

\bibitem{DGPT}
D.~Danielli, N.~Garofalo, A.~Petrosyan, and T.~To.
\newblock Optimal regularity and the free boundary in the parabolic signorini
  problem.
\newblock {\em Mem. Amer. Math. Soc.}, to appear.

\bibitem{DeG}
E.~De~Giorgi.
\newblock Sulla differenziabilit\`a e l'analiticit\`a delle estremali degli
  integrali multipli regolari.
\newblock {\em Mem. Accad. Sci. Torino. Cl. Sci. Fis. Mat. Nat. (3)}, 3:25--43,
  1957.

\bibitem{Evans}
L.~C. Evans.
\newblock {\em Partial differential equations}, volume~19 of {\em Graduate
  Studies in Mathematics}.
\newblock American Mathematical Society, Providence, RI, 1998.

\bibitem{LSU}
O.~A. Ladyzenskaja, V.~A. Solonnikov, and N.~N. Uralceva.
\newblock {\em Linear and quasilinear equations of parabolic type}.
\newblock Translated from the Russian by S. Smith. Translations of Mathematical
  Monographs, Vol. 23. American Mathematical Society, Providence, R.I., 1968.

\bibitem{Struwe}
M.~Struwe.
\newblock On the {H}\"older continuity of bounded weak solutions of quasilinear
  parabolic systems.
\newblock {\em Manuscripta Math.}, 35(1-2):125--145, 1981.

\bibitem{Widman}
K.-O. Widman.
\newblock H\"older continuity of solutions of elliptic systems.
\newblock {\em Manuscripta Math.}, 5:299--308, 1971.

\end{thebibliography}
\index{Bibliography@\emph{Bibliography}}%

\begin{tabular}{l}
Ioannis Athanasopoulos\\ University of Crete \\ Department of Mathematics  \\ 71409 \\
Heraklion, Crete GREECE
\\ {\small \tt athan@uoc.gr}
\end{tabular}
\begin{tabular}{l}
Luis Caffarelli\\ University of Texas \\ Department of Mathematics  \\ TX 78712\\
Austin, USA
\\ {\small \tt caffarel@math.utexas.edu}
\end{tabular}
\begin{tabular}{lr}
Emmanouil Milakis\\ University of Cyprus \\ Department of Mathematics \& Statistics \\ P.O. Box 20537\\
Nicosia, CY- 1678 CYPRUS
\\ {\small \tt emilakis@ucy.ac.cy}
\end{tabular}

\end{document}